\documentclass[10pt,a4paper,reqno]{amsart}%,leqno
\usepackage{amssymb,amsmath,amsthm,fancyhdr,graphicx,latexsym,mathabx}
\usepackage{charter}
\usepackage{enumitem}
\usepackage{xcolor}

\newtheorem{definition}{Definition}[section]
\newtheorem{thm}[definition]{Theorem}

\newtheorem{lem}[definition]{Lemma}
\newtheorem{cor}[definition]{Corollary}  
\newtheorem{prop}[definition]{Proposition}

\theoremstyle{plain}  
\newtheorem{remark}[definition]{Remark}

\numberwithin{equation}{section}
\numberwithin{figure}{section}

\newcommand{\eqand}{\ensuremath{\quad \textrm{ and } \quad}}

\newcommand{\eqkomma}{\ensuremath{\quad , \quad}}

\newcommand{\foot}{\footnote}

\newcommand{\lb}{\ensuremath{\left(}}
\newcommand{\rb}{\ensuremath{\right)}}
\newcommand{\equi}{\ensuremath{\Leftrightarrow}}
\newcommand{\follows}{\ensuremath{\Rightarrow}}

\newcommand{\ld}{\ensuremath{,\ldots,}}
\newcommand{\ssq}{\ensuremath{\subseteq}}
\newcommand{\smin}{\ensuremath{\setminus}}

\newcommand{\wh}{\ensuremath{\widehat}}

\newcommand{\ind}{\ensuremath{\mathbf{1}}}

\newcommand{\Id}{\ensuremath{\mathrm{Id}}}

\newcommand{\Per}{\ensuremath{\mathop{\rm Per}}}

\newcommand{\nfolge}[1]{\ensuremath{(#1)_{n\in\mathbb{N}}}}
\newcommand{\nofolge}[1]{\ensuremath{(#1)_{n\in\mathbb{N}_0}}}
\newcommand{\nzfolge}[1]{\ensuremath{(#1)_{n\in\mathbb{Z}}}}
\newcommand{\kfolge}[1]{\ensuremath{(#1)_{k\in\mathbb{N}}}}

\newcommand{\jfolge}[1]{\ensuremath{(#1)_{j\in\mathbb{N}}}}

\newcommand{\Matrix}[2]{\ensuremath{\left(\begin{array}{#1} #2
      \end{array}\right)}} 

\newcommand{\twovector}[2]{\ensuremath{\left(\begin{array}{c} #1 \\
        #2 \end{array}\right)}}

\newcommand{\alphlist}{\begin{list}{(\alph{enumi})}{\usecounter{enumi}\setlength{\parsep}{2pt}
      \setlength{\itemsep}{1pt} \setlength{\topsep}{5pt}
      \setlength{\partopsep}{3pt}}}
\newcommand{\arablist}{\begin{list}{(\arabic{enumi})}{\usecounter{enumi}\setlength{\parsep}{2pt}
          \setlength{\itemsep}{1pt} \setlength{\topsep}{5pt}
          \setlength{\partopsep}{3pt}}}
\newcommand{\romanlist}{\begin{list}{(\roman{enumi})}{\usecounter{enumi}\setlength{\parsep}{2pt}
              \setlength{\itemsep}{1pt} \setlength{\topsep}{5pt}
              \setlength{\partopsep}{3pt}}}
\newcommand{\Romanlist}{\begin{list}{(\Roman{enumi})}{\usecounter{enumi}\setlength{\parsep}{2pt}
              \setlength{\itemsep}{1pt} \setlength{\topsep}{5pt}
              \setlength{\partopsep}{3pt}}}
\newcommand{\bulletlist}{\begin{list}{$\bullet$}{\setlength{\parsep}{2pt}
                \setlength{\itemsep}{1pt} \setlength{\topsep}{5pt}
                \setlength{\partopsep}{3pt}\setlength{\leftmargin}{15pt}}} 
\newcommand{\Alphlist}{\begin{list}{(\Alph{enumi})}{\usecounter{enumi}\setlength{\parsep}{2pt}
      \setlength{\itemsep}{1pt} \setlength{\topsep}{5pt}
      \setlength{\partopsep}{3pt}}}
\newcommand{\listend}{\end{list}}

\newcommand{\N}{\ensuremath{\mathbb{N}}} 
\newcommand{\R}{\ensuremath{\mathbb{R}}}
\newcommand{\Z}{\ensuremath{\mathbb{Z}}}

\newcommand{\cA}{\ensuremath{\mathcal{A}}}

\newcommand{\cF}{\ensuremath{\mathcal{F}}}

\newcommand{\cO}{\ensuremath{\mathcal{O}}}
\newcommand{\cP}{\ensuremath{\mathcal{P}}}
\newcommand{\cQ}{\ensuremath{\mathcal{Q}}}
\newcommand{\cR}{\ensuremath{\mathcal{R}}}

\newcommand{\cW}{\ensuremath{\mathcal{W}}}

\newcommand{\nLim}{\ensuremath{\lim_{n\rightarrow\infty}}}

\newcommand{\kLim}{\ensuremath{\lim_{k\rightarrow\infty}}}

\newcommand{\nKonv}{\ensuremath{\stackrel{n\rightarrow
      \infty}{\longrightarrow}}}

\newcommand{\kKonv}{\ensuremath{\stackrel{k\rightarrow
      \infty}{\longrightarrow}}}

\newcommand{\nocup}{\ensuremath{\bigcup_{n\in\N_0}}}

\newcommand{\jnergsum}{\ensuremath{\sum_{j=0}^{n-1}}}

\newcommand{\ntel}{\ensuremath{\frac{1}{n}}}

\newcommand{\floor}[1]{\left\lfloor #1 \right\rfloor}

\newcommand\restr[2]{{% we make the whole thing an ordinary symbol
  \left.\kern-\nulldelimiterspace % automatically resize the bar with \right
  #1 % the function
  \littletaller % pretend it's a little taller at normal size
  \right|_{#2} % this is the delimiter
  }}

\newcommand{\littletaller}{\mathchoice{\vphantom{\big|}}{}{}{}}
\newcommand{\Ap}{\ensuremath{\mathrm{Ap}}}
\newcommand{\ap}{\ensuremath{\mathrm{ap}}}

\title{Permutative lifts of strictly ergodic subshifts} \author{Z.~Kang and T.~J\"ager}

\begin{document}

\begin{abstract} Permutative sliding block codes -- in the sense of Hedlund -- give rise to
  finite-to-one extensions of subshifts. We provide criteria under which minimality and unique
  ergodicity are preserved by this `lifting procedure'. These findings are illustrated by means of
  some natural example families, which we use to obtain strictly ergodic finite-to-one extensions of
  substitution subshifts (including the case of Fibonacci, Tribonacci, silver mean and noble means
  substitutions). We further study the interplay between permutative sliding block codes and Toeplitz
  flows. In particular, we find examples which demonstrate that a permutative lift of a strictly
  ergodic subshift may be minimal, but not uniquely ergodic -- a phenomenon which cannot occur for
  primitive substitution subshifts.
  
  
\end{abstract}

\maketitle
%\tableofcontents

\section{Introduction}

\subsection{Context and overview.} The study of extension structures has a long tradition in ergodic
theory and topological dynamics. Early seminal results include the structure theorems by Furstenberg
and Veech for minimal distal flows \cite{Furstenberg1963StructureTheorem,Veech1970PointDistalFlows}
and the Furstenberg-Zimmer Structure Theorem for ergodic measure-preserving transformations
\cite{Zimmer1976FurstenbergZimmerStructureTheorem,Furstenberg1977FurstenbergZimmerStructureTheorem}. Further
progress with a particular view towards number-theoretic applications has been made in
\cite{HostKra2005NonconventionalErgodicAverages,HostKraMaass2010NilsequencesStructureTheorem} (see
also \cite{HostKra2018NilpotentStructures}). In the context of entropy theory, symbolic extensions of
smooth dynamical systems have been studied in
\cite{DownarowiczNewhouse2005SymbolicExtensionsOfSmoothSystems,%
  DownarowiczMaas2009SymbolicExtensionsOfSmoothIntervalMaps,Burguet2011SymbolicExtensionsOfSurfaceDiffeomorphisms}. Apart
from these particular examples, there exists a wealth of further results on extension structures of
dynamical systems in the literature -- we restrict here to cite only to a biased selection
\cite{AuslanderGlasner1977DistalAndHighlyProximalExtensions,FurstenbergGlasner1978ExistenceOfIsometricExtensions,%
  ShaoYe2012RegionallyProximalRelationForMinimalSystems,GlasnerGutmanYe12018HigherOrderRegionallyProximalRelations,%
  GlasnerHuangShaoWeissYe2025TopologicalCharacteristicFactorsAndNilsystem,%
  QiuYe2025VeechTheoremOfHigherOrder}.

More recently, different topological and measure-theoretic concepts of finite extension structures and
their relations to dynamical properties have received considerable attention (e.g.\
\cite{LiTuYe2015MeanSensitivity,DownarowiczGlasner2015IsomorphicExtensionsAndMeanEquicontinuity,%
  Glasner2018MinimalTameSystems,%
  FuhrmannGlasnerJaegerOertel2021TameImpliesRegular,GarciaJaegerYe2021DiamMeanEquicontinuity,%
  BreitenbucherHauptJaeger2024FiniteTopomorphicExtensions,GarciaRamosLeon2024CoincidenceRankAndMultivariateEquicontinuity,%
  LiuXuZhang2025IndependenceMeanSensitivityAndMTameness}). At the
same time, multivariate versions of classical notions like stability, sensitivity, equicontinuity or
tameness have equally attracted significant interest
\cite{HuangLianShaoYe2021MinimalSystemWithFinitelyManyMeasures,LiYu2021MeanSensitiveTuples,LiYeYu2022EquicontinuityAndSensitivityInMeanForms,%
  LiuYin2023MeanSensitiveTuplesOfGroupActions,LiLiuTuYu2023SequenceEntropyTuples}. These are again
closely related to respective notions of finite extension structures
\cite{ShaoYeZhang2008NSensitivity,HuangLianShaoYe2021MinimalSystemWithFinitelyManyMeasures,%
  BreitenbucherHauptJaeger2024FiniteTopomorphicExtensions,GarciaRamosLeon2024CoincidenceRankAndMultivariateEquicontinuity,%
  Haupt2025multivariateFrequentStability,LiuXuZhang2025IndependenceMeanSensitivityAndMTameness,%
  Liu2026ConditionalTopomorphicDegreeaAndMultivariateMeanEquicontinuity}. From a broader viewpoint,
these developments contribute towards a structure theory of zero entropy systems in topological
dynamics.

In this context, the construction and description of examples is of vital interest. The problem is
that the main interest does not lie in arbitrary finite-to-one extension of a given system --
otherwise the direct product with any finite system would do -- but rather in those extensions that
keep essential dynamical properties of the underlying system. In particular, much of the
above-mentioned theory on finite extensions has been developed in the setting of minimal or even
strictly ergodic systems, so that finite extensions which preserve these properties are of particular
interest. However, the construction of minimal or strictly ergodic extensions is in general a
non-trivial problem and equally has a long history (e.g.\
\cite{anzai:1951,furstenberg:1961,GlasnerWeiss1979ConstructionOfMinimalSkewProducts,%
  Lemanczyk1987ErgodicZ2Extensions,Lemanczyk1988ToeplitzZ2Extensions,HermanPutnamSkau1992BratelliDiagrams,%
  bjerkloev:2005a,CortezPetite2008GOdometersAndExtensions,Jaeger2009SNAinCircleMaps,%
  DownarowiczGlasner2015IsomorphicExtensionsAndMeanEquicontinuity,DeeleyPutnamStrung2021NonhomogeneousExtensionsOfMinimalCantorSystems,%
  HauptJaeger2025SmoothIsomorphicExtensions}).\smallskip

In this article, we focus on a specific pathway that leads to a wealth of new examples of strictly
ergodic finite extensions of symbolic systems. The principal idea already goes back to the work of
Hedlund on sliding block codes \cite{Hedlund1969SlindingBlockCodes}, but it has apparently never been
fully exploited in this context so far. We consider the class of {\em permutative} sliding block
codes, which has been described in some detail in \cite[Sections 6 and
17]{Hedlund1969SlindingBlockCodes}. One of their structural properties is the fact that they are
always constant-to-one endomorphisms of the whole shift space. This means that they can be used to
lift any subshift to a finite-to-one extension by simply taking the preimage, which we will refer to
as a {\em permutative lift} of the subshift. However, it is a priori not clear to what extent
dynamical properties are preserved by this procedure. The main questions that we will focus on here
are the following.
\begin{itemize}
 \item  If $\varphi$ is a permutative sliding block code and $Y$ is a minimal subshift, is the extension
   $\varphi^{-1}(Y)$ minimal as well?
 \item If $\varphi$ is a permutative sliding block code and $Y$ is a uniquely ergodic subshift, is the
   extension $\varphi^{-1}(Y)$ uniquely ergodic as well?
 \end{itemize}
 It turns out that there is no general answer to these questions. In fact, on the one hand it is easy
 to see that even a fixed point of the shift can have non-minimal, and therefore also non-uniquely
 ergodic, lifts under permutative sliding block codes (a simple example will be given below). On the
 other hand, the Thue-Morse subshift, viewed as an extension of the period doubling subshift, provides
 a paradigmatic example in which strict ergodicity is preserved during this lifting
 procedure. Moreover, the answers to the above questions can be independent -- in
 Section~\ref{ToeplitzFlows}, we provide an example of a strictly ergodic Toeplitz flow whose lift
 under a simple permutative sliding block code is minimal, but not uniquely ergodic.

 Hence, our aim is to establish criteria -- both in terms of the subshift $Y$ and the sliding block
 code $\varphi$ -- which allow to ensure that the lift of a strictly minimal/uniquely ergodic subshift
 $Y$ under $\varphi$ is again minimal/uniquely ergodic. It turns out that the results we obtain hinge
 on the study of the behaviour of certain {\em return permutations} associated to a permutative
 sliding block code, which thus turn out to be a crucial concept in our context. We are convinced that
 these considerations will equally be instrumental in addressing a wealth of further questions that
 naturally occur in this setting, in particular concerning further dynamical spectral properties and
 quantities or the study of broader classes of examples. The latter may also serve as a useful testing
 ground in the context of the above-mentioned recent developments in the structure theory of
 topological dynamics.  \smallskip

 \subsection{Presentation of the main results}
 In order to give a more precise statement of the main concepts and results, we need to introduce some
 notation and terminology.  Given a finite alphabet $\cA_\kappa=\{0\ld \kappa-1\}$, $\kappa\in\N$, we
 call $\cA_\kappa^\Z$ the {\em shift space} over the alphabet $\cA_\kappa$ and define the {\em (left)
   shift map} $\sigma:\cA_\kappa^\Z\to\cA_\kappa^\Z$ by $\sigma(x)_n=x_{n+1}$ for all $n\in\Z$.  By
 $\cA^*_\kappa=\nocup \cA_\kappa^n$, we denote the set of all finite words with letters in
 $\cA_\kappa$ and write $|w|$ for the length of $w\in\cA^*_\kappa$. Given $y\in\cA_\kappa^\Z$ and
 integers $k<n$, we let $y_{[k,n]}=y_k\ldots y_n$.  A {\em subshift} $Y$ is a closed
 $\sigma$-invariant subset of $\cA^\Z_\kappa$ and $\cW(Y)$ denotes the set of all finite words that
 occur in sequences of $Y$.
 
 A continuous map $\varphi:\cA_\kappa^\Z\to\cA_\kappa^\Z$ is called a {\em shift endomorphism} if it
 commutes with the shift map, that is, $\varphi\circ\sigma=\sigma\circ\varphi$.  A {\em block code} of
 length $\ell\in\N$ is a map $h:\cA_\kappa^\ell\to\cA_\kappa$, and we denote by $\cF(\kappa,\ell)$ the
 family of such mappings.  Each block code $h$ induces a shift endomorphism $\eta_h$, referred to as a
 {\em sliding block code}, by
\[
\eta_h(x)_n \ = \ h\lb x_{[n,n+\ell-1]}\rb
\]
for all $n\in\Z$. Similarly, for each $m\in\N$, $h$ defines a map
$h_m:\cA^{m+\ell-1}_\kappa\to\cA^m_\kappa$ by $h_m(v)_n=h\lb v_{[n,n+\ell-1]}\rb$, $n=0\ld m-1$.  The
seminal Curtis-Hedlund-Lyndon Theorem states that, up to post-composition with a shift iterate, every
shift endomorphism is given by a sliding block code \cite{Hedlund1969SlindingBlockCodes}.

Following \cite[Section 6]{Hedlund1969SlindingBlockCodes}, we call a block code
$h\in \cF(\kappa,\ell)$ (and likewise its associated sliding block code $\eta_h$) {\em permutative} if
for all $u\in\cA_\kappa^{\ell-1}$ the mappings
\[
a \mapsto h(ua) \eqand a \mapsto h(au)
\]
are permutations of the alphabet $\cA_\kappa$. In this context, we will call elements of
$\cA_\kappa^{\ell-1}$ {\em prefixes} or {\em suffixes}, depending on the situation.  An important
observation is the following: given a permutative block code $h\in\cF(\kappa,\ell)$, a prefix
$u\in\cA_\kappa^{\ell-1}$ and an arbitrary sequence $y\in\cA_\kappa^\Z$, the permutative structure of
$h$ allows to construct -- via a straightforward induction -- a unique sequence $\xi^u(y)$ with prefix
$u$ and image $y$ under $\eta_h$. More precisely, we have
\begin{prop}[{\cite{Hedlund1969SlindingBlockCodes}}]
     \label{p.intro_inverse_branches} Let $h\in\cF(\kappa,\ell)$ be permutative.
     Then for every $u\in\cA_\kappa^{\ell-1}$ there exists a continuous mapping
     \[
       \xi^u:\cA_\kappa^\Z\to\cA_\kappa^\Z\quad , \quad y\mapsto \xi^u(y)
     \]
     such that $\xi^u(y)_{[0,\ell-2]}=u$ and $\eta_h\lb\xi^u(y)\rb=y$ holds for all
     $y\in \cA_{\kappa}^\Z$. Further, we have
      \begin{equation} \label{e.fibres_via_xi_u}
     \eta_h^{-1}(y) \ = \ \left\{\xi^u(y)\mid u\in\cA_\kappa^{\ell-1}\right\} \ , 
      \end{equation}
      and consequently $\#\eta^{-1}_h(y)=\kappa^{\ell-1}$ for all $y\in\cA_\kappa^\Z$.

      Moreover, if $u\neq u'\in\cA^{\ell-1}_\kappa$, the sequences $\xi^u(y)$ and $\xi^{u'}(y)$ are
      $(\ell-1)$-separated, meaning that they do not coincide on any interval of $\ell-1$ consecutive
      integers.
    \end{prop}
    The mappings $\xi^u$ will be called {\em (continuous) inverse branches} of $\eta_h$.\smallskip

    We should point out that some of these assertions are not explicitly stated in
    \cite{Hedlund1969SlindingBlockCodes}, while others are obtained as consequences of more general
    results on surjective sliding block codes. We therefore provide a short direct proof in
    Section~\ref{InverseBranches}.\smallskip

    The inductive construction of the sequences $\xi^u(y)$ also works on a finite level, so that for
    each finite word $w\in\cA^*_\kappa$ and each prefix $u\in\cA^{\ell-1}_\kappa$ there exists a
    unique word $\xi^u(w)\in\cA_\kappa^{|w|+\ell-1}$ such that $\xi^u(w)_{[0,\ell-2]}=u$ and
    $h_{|w|}\lb\xi^u(w)\rb = w$. This allows to associate, to any given permutative block code $h$,
    the following family of permutations on the set of prefixes. Given a finite word $w\in\cA^*$, we
    define
    \[
      \Phi^{(h)}_w : \cA^{\ell-1}_\kappa\to\cA^{\ell-1}_\kappa \eqkomma u \mapsto
      \xi^u(w)_{[|w|,|w|+\ell-2]} \ 
    \]
    and call this the {\em prefix permutation induced by $w$}.  In other words, $\Phi^{(h)}_w(u)$ is
    defined as the suffix of the unique lift $\xi^u(w)$ of the word $w$ with prefix $u$.

    \begin{prop}\label{p.intro_prefix_permutations}
      Suppose that $h\in\cF(\kappa,\ell)$ is permutative. Then for every $w\in\cA^*_\kappa$, the
      mapping $\Phi_w^{(h)}$ is a permutation of $\cA_{\kappa}^{\ell-1}$.  Moreover, given
      $w,v\in\cA^*_\kappa$, we have
      \[
        \Phi^{(h)}_{wv} \ = \  \Phi^{(h)}_v\circ \Phi^{(h)}_w \ .
      \]
    \end{prop}
    The proof is given in Section~\ref{InverseBranches}.  Altogether, the above concepts lead to a
    useful structural result, namely a skew-product representation for permutative lifts of subshifts.

    \begin{thm}\label{t:skew_product_representation}
      Suppose that $h\in \cF(\kappa,\ell)$ is a permutative block code, $Y\ssq \cA_\kappa^\Z$ is a
      subshift and $X=\eta_h^{-1}(Y)$. Let
  \begin{equation}\label{e.conjugacy_to_skew_product}
     \gamma : X \to Y\times\cA_\kappa^{\ell-1}\quad , \quad x \mapsto \left(\eta_h(x),x_{[0,\ell-2]}\right) 
  \end{equation}
   and 
  \begin{equation}\label{e.skew_product_system}
      T : Y\times \cA_\kappa^{\ell-1}\to Y\times\cA_\kappa^{\ell-1}\quad , \quad (y,u)\mapsto
      \left(\sigma(y),\Phi_{y_0}(u)\right) \ .
  \end{equation}
  Then $\gamma$ and $T$ are homeomorphisms and $\gamma$ is a conjugacy between $(X,\sigma)$ and
  $(Y\times\cA_\kappa^{\ell-1},T)$.
\end{thm}

A particularly significant case are prefix permutations induced by return words: we call
    \[
      \cR^Y_h(w) \ = \ \left\{ \Phi_{wv} \mid wvw\in \cW(Y)\right\}
    \]
    the set of {\em return permutations} for the word $w$.
    \begin{remark} \label{r.alternative_return_permutation_sets}
      We note that for all our purposes, the set $\cR^Y_h(w)$ could also be replaced by the possibly
      larger set
      \[
        \tilde\cR^Y_h(w)\ = \ \left\{\Phi_v\mid vw\in \cW(Y) \textrm{ and } vw \textrm{ starts with }
          w\right\} \ ,
      \]
      or by the subset of $\tilde\cR^Y_h(y)$ given by
      \[
        \wh \cR^Y_h(w) \ = \ \left\{\Phi_v\ \left| \begin{array}{c} vw\in \cW(Y) \textrm{ and } vw \textrm{ starts with } w\\
                                                     \textrm{ and contains exactly 2 occurrences of } w  \end{array}\right.  \right\} \ .
                                             \]
                                             The latter is the set of {\em return words} of $w$ in the
                                             sense of \cite{Durand}.  All the results below still hold
                                             when $\cR^Y_h(w)$ is replaced by either of these sets,
                                             and the proofs remain valid with only minor technical
                                             modifications.
    \end{remark}
    Further, we say a subset $P$ of the permutation group $\Per\lb\cA^{\ell-1}_\kappa\rb$ {\em acts
    transitively on $\cA^{\ell-1}_\kappa$} if this holds for the subgroup $\langle P\rangle$ generated
    by $P$.  These notions now allow to state equivalent criteria for the minimality of permutative
    lifts.

    \begin{thm}\label{t:intro_minimality}
 Suppose that $h\in \cF(\kappa,\ell)$ is a permutative block code and $Y\subseteq \cA_\kappa^{\Z}$ is
 a minimal subshift. Let $X=\eta_h^{-1}(Y)$. Then the following are equivalent:
        \begin{itemize}
        \item[(i)] $(X,\sigma)$ is minimal;
        \item[(ii)] for every word $w\in\cW(Y)$, the return permutation set $\cR(w)$ acts transitively
          on $\cA_\kappa^{\ell-1}$;
        \item[(iii)] there exists a sequence of words $w^{(n)}\in\cW(Y)$ such that
          $|w^{(n)}|\nKonv\infty$ and for all $n\in\N$ the return permutation set $\cR(w^{(n)})$ acts
          transitively on $\cA_\kappa^{\ell-1}$;
        \item[(iv)] there exists a sequence of words $w^{(n)}\in\cW(Y)$ such that $w^{(n+1)}$ starts
          with $w^{(n)}$, $|w^{(n)}|\nKonv\infty$ and for all $n\in\N$ the family
          $\bigcup_{k\geq n} \cR(w^{(k)})$ acts transitively on $\cA_\kappa^{\ell-1}$.
        \end{itemize}
\end{thm}

The proof is given in Section~\ref{Minimality_via_return_permutations}. As we will see below, the last
formulation in (iv) is tailor-made for the application to substitution subshifts. \smallskip

In order to ensure the unique ergodicity of permutative lifts, there are actually two
possibilities. First, we establish a general criterion, similar in spirit to
Theorem~\ref{t:intro_minimality}, based on the concept of return permutations. This will be used in
the study of (permutative lifts of) Toeplitz flows in Section~\ref{ToeplitzFlows}. Since its
formulation requires some additional notation and technical preparations, we refer to
Section~\ref{UniqueErgodicity} for the precise statement. For the special case of primitive
substitution subshifts, it turns out that one can alternatively show that whenever a permutative lift
is minimal, it is linearly recurrent (and hence uniquely ergodic
\cite{Durand2000LinearRecurrence}\cite{Bruin2022SymbolicDynamics}). While a general proof of this fact
will be given in \cite{GohlkeJaegerKang2026PermutativeSubstitutionLifts}, we provide a short direct
argument for the case of the noble means substitutions in
Section~\ref{UniqueErgodicityForSubstitutions}.

In order to illustrate our findings and to demonstrate that the above minimality criteria can be
applied efficiently, we consider a natural two-parameter family of permutative block codes, given by
\[
h_{\kappa,\ell} : \cA_\kappa^{\ell}\to\cA_\kappa\quad , \quad x_0\ldots x_{\ell-1}\mapsto x_{\ell-1}-x_0\bmod \kappa \ , 
\]
with integer parameters $\kappa,\ell\geq 2$. For simplicity, we write $\eta_{\kappa,\ell}$ instead of
$\eta_{h_{\kappa,\ell}}$.  We then apply these to obtain permutative lifts of some classical examples
of substitution subshifts. Recall that the noble means substitution with parameter $p\in\N$ on the
alphabet $\cA_2=\{0,1\}$ is given by the substitution rule $0\mapsto 0^p1,\ 1\mapsto 0$. It includes
the Fibonacci substitution for $p=1$ and is conjugate to the silver mean substitution for $p=2$
\cite[Section 4.4]{BaakeGrimm2013AperiodicOrder}.

\begin{thm}
  \label{t:intro_fibonacci}
  Let $X_{\mathrm{nm},p}$ denote the subshift generated by the noble means substitution with parameter
  $p\in\N$. Then $\eta_{\kappa,\ell}^{-1}\lb X_{\mathrm{nm},p}\rb$ is strictly ergodic for all
  $\kappa,\ell\geq 2$.
\end{thm}

The statement remains valid when the noble means substitution is replaced by the Tribonacci
substitution (where $\kappa\geq 3$ is required for Tribonacci). The details are given in
Sections~\ref{SubstitutionApplications} (for minimality) and \ref{UniqueErgodicityForSubstitutions}
(for unique ergodicity).  A very interesting case is the period doubling substitution, induced by the
substitution rule $\Gamma_{\mathrm{pd}}:\cA_2\mapsto\cA_2^*$, $1\mapsto 10$, $0\mapsto 11$. In this
case, the minimality of the lift depends on the values of $\kappa$ and $\ell$.

\begin{thm}
  \label{t:intro_period_doubling}
  Let $X_\mathrm{pd}$ denote the period doubling subshift. Then the lift
  $X=\eta_{\kappa,\ell}^{-1}(X_\mathrm{pd})$ is strictly ergodic if $\ell\geq 2$ is even and if
  $\ell=3$ and $\kappa$ is odd. In all other cases, $X$ is non-minimal.
\end{thm}

We note that this statement covers a very classical example, which is the factor relation between the
period doubling subshift $X_\mathrm{pd}$ and the Thue-Morse subshift $X_\mathrm{TM}$. In fact, the
latter coincides with the lift of $X_\mathrm{pd}$ under the sliding block code $\eta_{2,2}$. Moreover,
lifts of Toeplitz subshifts via $\eta_{2,2}$ have been studied in the context of $\Z_2$-extensions and
their spectral theory (e.g.\
\cite{Lemanczyk1987ErgodicZ2Extensions,Lemanczyk1988ToeplitzZ2Extensions,FilipowiczKwiatkowskiLemanczyk1988ApproximationOfZ2Cocycles}).
We also note that $\eta_{2,2}^2=\eta_{2,3}$. Hence, the lift of the Thue-Morse subshift under
$\eta_{2,2}$ coincides with $\eta^{-1}_{2,3}\lb X_\mathrm{pd}\rb$ and is therefore not minimal by
Theorem~\ref{t:intro_period_doubling}. An easier example of a subshift with a non-minimal lift
under $\eta_{2,2}$ is given by the fixed point $0^\Z$, with preimages $0^\Z$ and~$1^\Z$.

The prefix permutations for $h_{2,2}$ are particularly easy to determine: we have
\[
  \Phi_w^{h_{2,2}} : \cA_2\to\cA_2 \quad , \quad a \mapsto a + \sum_{j=0}^{|w|-1} w_j \ \bmod 2 \ .
\]
This entails the following simplified minimality criterion for the special case of $h_{2,2}$.

\begin{cor} \label{c.intro_minimality_criterion} Let $Y\ssq \cA_2^\Z$ be a minimal subshift. Then
  $X=\eta_{2,2}^{-1}(Y)$ is minimal if and only if there exists sequences of words
  $w^{(n)},v^{(n)}\in\cW(Y)$ such that $|w^{(n)}|\nKonv \infty$, $w^{(n)}v^{(n)}w^{(n)}\in\cW(Y)$ and
  $\sum_{j=0}^{|w^{(n)}v^{(n)}|-1} (w^{(n)}v^{(n)})_j$ is odd for all $n\in\N$.
\end{cor}

A more detailed discussion is given in Section~\ref{ExampleFamily}.  Finally, we study the interplay
between permutative sliding block codes and Toeplitz subshifts in Section~\ref{ToeplitzFlows}. The
main results can be summarised as follows.

\begin{thm}\label{t.intro_toeplitz_results}
  \alphlist
\item[(a)] For any permutative block code $h$, there exists a regular Toeplitz subshift $Y$ such that $\eta_h^{-1}(Y)$
  is strictly ergodic.
\item[(b)] For any permutative block code $h$, there exists a regular Toeplitz subshift $Y$ such that $\eta^{-1}_h(Y)$
  is not minimal.
\item[(c)] For any Toeplitz subshift $Y$, there exists a permutative block code $h$ on the same
  alphabet such that $\eta_h^{-1}(Y)$ is not minimal.
\item[(d)] There exists a regular Toeplitz subshift $Y$ such that $\eta_h^{-1}(Y)$ is not minimal for
  all permutative block codes $h$. 
\item[(e)] There exists a regular Toeplitz subshift $Y$ and a permutative block code $h$ such that
  $\eta_h^{-1}(Y)$ is minimal, but not uniquely ergodic. \listend
  \end{thm}
  These assertions are restated and proven below as (a) Theorem~\ref{t:unique_ergodic_lift}, (b)
  Proposition~\ref{p.h_fixed_non-minimal_Toeplitz}, (c) Theorem~\ref{t.Toeplitz_fixed_h_non-minimal},
  (d) Theorem~\ref{t.Toeplitz_with_universally_non-minmal_lifts} and (e)
  Theorem~\ref{t.Toeplitz_with_minimal_but_non-uniquely_ergodic_lift}. We note that in contrast to
  (c), there exist strictly ergodic subshifts whose permutative lifts are strictly ergodic for all
  permutative sliding block codes on the same alphabet. This is true, for instance, for all subshifts
  induced by noble means substitutions \cite{GohlkeJaegerKang2026PermutativeSubstitutionLifts}.
  \medskip

The article is organised as follows: in Section~\ref{Prelim}, we provide the required background
knowledge on symbolic and topological dynamics and sliding block codes. In
Section~\ref{InverseBranches}, we consider inverse branches and induced prefix permutations of
permutative sliding block codes and provide a proof of the skew product representation in
Theorem~\ref{t:skew_product_representation}.  In Section~\ref{Minimality}, we prove the equivalence of
the minimality criteria in Theorem~\ref{t:intro_minimality} and apply these to the lifts of
substitution systems under the parameter family $\eta_{\kappa,\ell}$ introduced above. The question of
unique ergodicity is treated in Section~\ref{UniqueErgodicity}. Finally, in
Section~\ref{ToeplitzFlows} we take a look at Toeplitz flows and their lifting properties in order to
prove the assertions stated in Theorem~\ref{t.intro_toeplitz_results}.  \medskip

{\bfseries Acknowledgments.}\quad We would like to thank Ela Krawczyk and Henk Bruin as well as
Philipp Gohlke for helpful discussions in general and the suggestion to consider linear recurrence in
our context in particular. This work was supported by the DFG-project 552775134 {\em `Finite
  topomorphic extensions of equicontinuous systems'} as well as the Chinese Scholarship Council (Grant
No.202308080192).\medskip

%%%%%%%%%%%%%%%%%%%%d%%%%%%%%%%%%%%%%%%%%%%%%%%%%%%%%%%%%%%%%%%%%%%%%%%%%%%%%%%%%%%%%%%%%
%%%%%%%%%%%%%%%%%%%%%%%%%%%%%%%%%%%%%%%%%%%%%%%%%%%%%%%%%%%%%%%%%%%%%%%%%%%%%%%%%%%%%%%%
%%%%%%%%%%%%%%%%%%%%%%%%%%%%%%%%%%%%%%%%%%%%%%%%%%%%%%%%%%%%%%%%%%%%%%%%%%%%%%%%%%%%%%%%

  \section{Preliminaries} \label{Prelim}

  \subsection{Shift spaces} \label{Prelim:shift_spaces} As before, given $\kappa\in\N$ we let
  $\cA_\kappa=\{0\ld \kappa-1\}$ and $\cA_\kappa^*=\nocup \cA_\kappa^n$, where $\cA_\kappa^0=\{e\}$,
  where $e$ is the empty word. We further define $|w|$ as the length of the word $w\in\cA_\kappa^*$,
  so that $w=w_0\ldots w_{|w|-1}$. Given two words $w,v$ with $|w|\leq |v|$, we write $w\triangleleft v$
  if there exists $k\in\{0\ld |v|-|w|\}$ such that $w_j=v_{k+j}$ for all $j=0\ld |w|-1$ and say {\em
    $w$ is contained in $v$} in this case. The same notation and terminology is used when $v$ is an
  infinite (one- or two-sided) sequence. The concatenation of two words $w,v\in\cA_\kappa^*$ is
  written as $wv$, and the same notation is used if $v$ is an infinite one-sided sequence. The
  concatenation of $p$ copies of a finite word $w$ will be denoted by $w^p$. Given $k<m\in\Z$ and
  $x\in\cA_\kappa^\Z$, we write $x_{[k,m]}=x_k\ldots x_m$. Analogous notation will be used for
  finite words in $\cA_\kappa^*$ or one-sided sequences in $\cA_\kappa^\N$.  Given any subset
  $Y\ssq \cA_\kappa^\Z$, we let
  \[
    \cW(Y) \ = \ \{w\in\cA_\kappa^* \mid \exists y\in Y: w
    \triangleleft y\} \ .
  \]
  and $\cW_n(Y)=\cW(Y)\cap \cA_\kappa^n$ for $n\in\N$. If $Y=\{\xi\}$, we write $\cW(\xi)$ instead of
  $\cW(\{\xi\})$. 

  We equip the shift space $\cA_\kappa^\Z$ with the metric
  $d(x,y) = e^{-\min\{|n|\mid n\in\Z,\ x_n\neq y_n\}}$, which makes it a compact and
  totally disconnected metric space. Given a finite word $w \in \cA_\kappa^{\{-n\ldots n\}}$, we
  define the corresponding {\em cylinder set of level $n$} as
  \[
       [w]_n \ = \ \left\{x\in\cA_\kappa^\Z\mid x_k=w_k \textrm{ for all } k=-n\ld n\right\}. 
  \]
  Similarly, given $x\in\cA_\kappa^\Z$, we let $[x]_n=[x_{-n}\ldots x_n]_n$. Note that thus
  $[x]_n=B_{e^{-n}}(x)$, where as usual $B_\delta(x)$ denotes the open ball of radius $\delta$ around
  $x$. More generally, we also consider finite sets $M\ssq\Z$ of positions and `words' $w\in\cA^M$,
  and then write
  \[
    [w]_M\ = \ \left\{x\in\cA^\Z\mid x_k=w_k \textrm{ for all } k\in M \right\} \ .
    \]
  Note that thus $[w]_n=[w]_{\{-n\ld n\}}$ (and in particular $[w]_n\neq [w]_{\{n\}}$ if $n\neq 0$).
  Finally, given $w\in\cA^n_\kappa$, we will also use the notation
  \[
    [w]_n^+ \ = \ [w]_{\{0\ld n-1\}} \ = \ \left\{x\in\cA^\Z_\kappa\mid x_k=w_k \textrm{ for all }
    k=0\ld n-1\right\} \ .
      \]
      By $\sigma:\cA_\kappa^\Z\to\cA_\kappa^\Z, \ \nzfolge{x_n} \mapsto \lb x_{n+1}\rb_{n\in\Z}$, we
      denote the (left) shift map on $\cA_\kappa^\Z$. For any collection of words
      $W\ssq \cA_\kappa^*$, we let
  \begin{equation}\label{e.permitted_words_subshift}
  X_W \ = \ \left\{x\in\cA^\Z_\kappa\mid \cW(x)\subseteq W\right\}
  \end{equation}
  and also write $X_\xi$ instead of $X_{W(\xi)}$. This obviously defines a (possibly empty) subshift
  \mbox{in $\cA^\Z_\kappa$.}  

  %%%%%%%%%%%%%%%%%%%%%%%%%%%%%%%%%%%%%%%%%%%%%%%%%%%%%%%%%%%%%%%%%%%%%%%%%%%%%%%%%%%%%%%%

  \subsection{Shift endomorphisms as factor maps between subshifts} \label{Prelim:shift_endomorphisms}
  A continuous mapping $\eta:\cA_\kappa^\Z\to\cA_\kappa^\Z$ that satisfies
  $\sigma\circ \eta=\eta\circ\sigma$ is called a {\em shift endomorphism}. For any $\ell\in\N$, we
  denote the set of {\em block codes} of length $\ell$ by
\[
\cF(\kappa,\ell) \ = \ \left\{h:\cA_\kappa^\ell\to \cA_\kappa\right\} \ .
\]
Given $m\geq 1$, any $h\in\cF(\kappa,\ell)$ induces mappings
$h_m:\cA_\kappa^{m+\ell-1}\to\cA_\kappa^m$ and a map $\eta_h:\cA_\kappa^\Z\to\cA_\kappa^\Z$, which are
given by
\begin{eqnarray}
  \label{e.finite_block_code_extension} h_m(w)_j & = & h\lb w_{[j,j+\ell-1]}\rb \quad \textrm{ for } j=0\ld
  m-1 \ , \\ \eta_h(x)_j & = & h\lb x_{[j,j+\ell-1]}\rb \quad \ \textrm{ for } j\in\Z
  \ ,\label{e.infinite_block_code_extension}
\end{eqnarray}
where $w\in\cA_\kappa^{m+\ell-1}$ and $x\in\cA_\kappa^\Z$, respectively. The map $\eta_h$ will be referred to as the
the {\em sliding block code} induced by $h$.

A closed $\sigma$-invariant subset $X\ssq\cA_\kappa^\Z$ is called a {\em subshift}.  Given
$\ell\in\N$, we denote the collection of all maps $h:\cW_\ell(X)\to \cA$ by
$\cF_X(\kappa,\ell)$. Then, similar to above, any $h\in\cF_X(\kappa,\ell)$ induces mappings
\[
  h_m:\cW_{m+\ell-1}(X)\to \cA_\kappa^m \eqand \eta^X_h:X\to\cA_\kappa^\Z \ ,
\]
which are again defined by (\ref{e.finite_block_code_extension}) and
(\ref{e.infinite_block_code_extension}), respectively. If we let $Y=\eta^X_h(X)$, then $\eta^X_h$ is a
factor map from $(X,\sigma)$ to $(Y,\sigma)$. Recall here that $\eta:X\to Y$ is a {\em factor map} if
$\eta$ is continuous and surjective and $\sigma\circ \eta=\eta\circ \sigma$. Note also that any block
code $h\in\cF_X(\kappa,\ell)$ can be extended to a block code in $\cF(\kappa,\ell)$, since the images
of any words in $\cA_\kappa^\ell\smin \cW_\ell(X)$ can be chosen freely in such an extension.

The following proposition is a version of the classical Curtis-Hedlund-Lyndon Theorem
\cite{Hedlund1969SlindingBlockCodes} restricted to subshifts (and the original statement is recovered
when $X=Y=\cA^\Z$). The proof does not differ from \cite{Hedlund1969SlindingBlockCodes} in any
essential way, but it is included for the convenience of the reader.
 \begin{thm}[Curtis-Hedlund-Lyndon]
   \label{p:hedlund_for_subshifts} Suppose that $X,Y\ssq\cA_\kappa^\Z$ are subshifts and $\eta:X\to Y$
   is a factor map. Then there exist $k\in\Z$, $\ell\in\N$ and $h\in\cF_X(\kappa,\ell)$ such that
   $\eta=\sigma^{-k}\circ\eta^X_h$.
 \end{thm}
 
 Note that if $\varphi$ is a factor map between the subshifts $(X,\sigma)$ and $(Y,\sigma)$, then so
 is $\eta=\sigma^k\circ\varphi$. Hence, we will always assume without loss of generality that $k=0$.
 \proof\quad As $\varphi$ is continuous, so is the map
 $\varphi_0:X\to\cA_\kappa,\ x\mapsto \varphi(x)_0$. Since $\varphi_0$ is finite-valued, it must be
 locally constant. Due to compactness, this means that $X$ can be written as a finite union of
 cylinder sets $C_1\ld C_n$ such that $\varphi_0$ is constant on each $C_i$. If $k$ is the maximal
 level of these cylinder sets, then $\varphi_0(x)=\varphi(x)_0$ only depends on $x_{-k}\ld
 x_k$. Hence, there exists a mapping $h:\cA_\kappa^{2k+1}\to\cA_\kappa$ such that
 $\varphi(x)_0=h\lb x_{-k}\ldots x_k\rb=h\lb \sigma^{-k}(x)_{[0,2k+1]}\rb$. The fact that $\varphi$
 commutes with the shift $\sigma$ now yields $\varphi(x)_n=h\lb\sigma^{-k}(x)_{[n,n+2k+1]}\rb$, which
 means that $\varphi=\eta_h\circ\sigma^{-k}=\sigma^{-k}\circ\eta_h$ as claimed.  \qed\medskip

A subshift $Y$ is called {\em linearly recurrent} or {\em linearly repetitive} if there exists some
constant $L\in\N$ such that $w\triangleleft v$ holds for all $w,v\in\cW(Y)$ whenever $|v|\geq L|w|$.

%%%%%%%%%%%%%%%%%%%%%%%%%%%%%%%%%%%%%%%%%%%%%%%%%%%%%%%%%%%%%%%%%%%%%%%%%%%%%%%%%%%%%%%%

\subsection{Topological dynamics} Throughout this article, a {\em topological dynamical system (tds)}
will be a pair $(X,T)$, where $X$ is a compact metric space and $T:X\to X$ a homeomorphism. The {\em
  orbit} of a point $x\in X$ under $T$ is given by $\cO_T(x)=\{T^n(x)\mid n\in\Z\}$. Similarly, the
{\em forward} and {\em backward} orbits of $x\in X$ are defined as $\cO_T^+(x)=\{T^n(x)\mid n\geq 0\}$
and $\cO_T^-(x)=\{T^n(x)\mid n\leq 0\}$, respectively. The tds $(X,T)$ is called {\em minimal} if all
orbits are dense, that is, $\overline{\cO_T(x)}=X$ for all $x\in X$. In this case, all forward and
backward orbits are dense as well.  In general, a point $x\in X$ is called {\em almost periodic} (with
respect to $T$) if $\lb\overline{\cO_T(x)},T\rb$ is minimal. Note that when $X$ is a subshift, then
almost periodicity of a sequence $x\in X$ is equivalent to the fact that every subword
$w\triangleleft x$ occurs infinitely often in $x$, with uniformly bounded gaps between successive
occurrences.

A tds $(X,T)$ is called {\em uniquely ergodic} if there exists a unique $T$-invariant probability
measure $\mu$ on $X$. This is equivalent to the fact that
\begin{equation}\label{e.unique_ergodicity}
  \nLim \jnergsum f\circ T^j(x) \ = \ \int_X f\ d\mu
\end{equation}
holds for all $x\in X$. In this case, the convergence in (\ref{e.unique_ergodicity}) is uniform in
$x\in X$. A tds $(X,T)$ is called {\em strictly ergodic} if it is both minimal and uniquely
ergodic. We note that this holds for all linearly recurrent subshifts. 

%%%%%%%%%%%%%%%%%%%%%%%%%%%%%%%%%%%%%%%%%%%%%%%%%%%%%%%%%%%%%%%%%%%%%%%%%%%%%%%%%%%%%%%%
%%%%%%%%%%%%%%%%%%%%%%%%%%%%%%%%%%%%%%%%%%%%%%%%%%%%%%%%%%%%%%%%%%%%%%%%%%%%%%%%%%%%%%%%
%%%%%%%%%%%%%%%%%%%%%%%%%%%%%%%%%%%%%%%%%%%%%%%%%%%%%%%%%%%%%%%%%%%%%%%%%%%%%%%%%%%%%%%%

\section{Permutative block codes: inverse branches and skew product form} \label{InverseBranches} As
mentioned, a block code $h\in\cF(\kappa,\ell)$ is called {\em permutative} if the two mappings
  \begin{eqnarray}
    \rho_u^{(h)} & : & \cA_\kappa \rightarrow \cA_\kappa \quad , \quad a\mapsto h(ua) \\ \lambda_u^{(h)} & : & \cA_\kappa
    \rightarrow \cA_\kappa \quad , \quad a \mapsto h(au)
  \end{eqnarray}
  are both permutations of the alphabet $\cA_\kappa$ for all $u\in\cA_\kappa^{\ell-1}$. when no
  ambiguity can arise, we will drop the reference to $h$ and simply write $\rho_u,\ \lambda_u$ . The
  block code $h$ is called {\em right permutative} if the $\rho_u$ are permutations for all
  $u\in\cA_\kappa^{\ell-1}$, and {\em left permutative} if this holds for the
  $\lambda_u$. \cite[Section 17]{Hedlund1969SlindingBlockCodes} shows that $h\in\cF(\kappa,\ell)$ is
  permutative if and only if $\#\eta^{-1}_h(y)=\kappa^{\ell-1}$. Non-permutative sliding block codes
  can be constant-to-one as well, but in this case the fibre cardinality must be strictly smaller. An
  example is given in \cite[Section 6]{Hedlund1969SlindingBlockCodes}.

  Now, suppose that $h$ is permutative, $w\in\cA_\kappa^m$ for some $m\geq 1$, $a\in\cA_\kappa$ and
  $v\in h_m^{-1}(w)$. If $v_{[m,m+\ell-2]} = u$ then $v\rho_u^{-1}(a)$ is the unique word of length
  $m+\ell$ that is mapped to $wa$ by $h_{m+1}$. Likewise, if $v_{[0,\ell-2]}=u$, then
  $\lambda_u^{-1}(a)v$ is the unique word of length $m+\ell$ which is mapped to $aw$. Using this fact,
  a straightforward induction shows that for every finite word $w\in\cA_\kappa^*$ and
  $u\in\cA_\kappa^{\ell-1}$ there exist unique words $\xi^u(w),\zeta^u(w)\in\cA_\kappa^{|w|+\ell-1}$
  such that
 \[
   h_{|w|}(\xi^u(w))\ = \ h_{|w|}(\zeta^u(w))\ = \ w
 \]
 and $\xi^u(w)$ starts with $u$, whereas $\zeta^u(w)$ ends with $u$. Similarly, for any sequence
 $y\in\cA_\kappa^\Z$ and a prefix $u\in\cA_\kappa^{\ell-1}$, there exists a unique sequence
 $\xi^u(y)\in\cA_\kappa^\Z$ such that
   \begin{eqnarray}\label{def:xi^u}
     \xi^u(y)_{[0,\ell-2]} \ = \ u \eqand \eta_h\lb\xi^u(y)\rb \ = \ y \ . 
   \end{eqnarray}
   
   Given a permutative block code $h\in \cF(\kappa,\ell)$ and $w\in \cA_\kappa^*$, this allows to
   define the {\em prefix permutation induced by $w$} as 
   \begin{equation} \label{e.def_prefix_permutation} \Phi_w^{(h)}: \cA_\kappa^{\ell-1}\rightarrow
     \cA_\kappa^{\ell-1} \ , \ u \mapsto \xi^u(w)_{[|w|,|w|+\ell-2]} \ . 
\end{equation}
In the same way, we define the {\em suffix permutation induced by $w$} as
\begin{equation}\label{e.def_inverse_prefix_permutation}
 \Psi^{(h)}_w : \cA_\kappa^{\ell-1}\to\cA_\kappa^{\ell-1}\quad , \quad u \mapsto
 \zeta^u(w)_{[0,\ell-2]} \ .
\end{equation}
As a direct consequence of the definitions, we obtain that $\Phi_w(\Psi_w(u))=u$ and
$\Psi_w(\Phi_w(u))=u$ for all $u\in\cA_\kappa^{\ell-1}$. Hence, these maps are inverse to each other
and must therefore be permutations of $\cA_{\kappa}^{\ell-1}$. We now prove
Propositions~\ref{p.intro_inverse_branches} and \ref{p.intro_prefix_permutations} stated in the
introduction.

\proof[\textit{\bfseries Proof of Proposition~\ref{p.intro_prefix_permutations}}] The fact that the mappings $\Phi_w$
are permutations already follows from the discussion above. It remains to show that
\begin{equation} \label{e.induced_permutation_composition}
  \Phi_{wv} \ = \ \Phi_v\circ\Phi_w 
\end{equation}
holds for all $w,v\in\cA^*_\kappa$. To that end, note that for any $u\in\cA_\kappa^{\ell-1}$, we have
that
\begin{equation}\label{e.xi_u_concatenation}
  \xi^u(wv) \ = \ \xi^u(w)_{[0,|w|-1]}\xi^{\Phi_w(u)}(v) \ . 
\end{equation}
Since $\xi^u(w)_{[0,|w|-1]}\xi^{\Phi_w(u)}(v)$ ends with $\Phi_v(\Phi_w(u))$, this proves
(\ref{e.induced_permutation_composition}).  \qed\medskip

As mentioned, the assertions of Proposition~\ref{p.intro_inverse_branches} follow from more general
results in \cite{Hedlund1969SlindingBlockCodes}, but we include a short direct argument for the
permutative case.

\proof[\textit{\bfseries Proof of
Proposition~\ref{p.intro_inverse_branches}}] \quad The existence of the mappings $\xi^u$ already
follows from the discussion above. Continuity is a consequence of the fact that
$\xi^u(y)_{[-n,n+\ell-1]}=\xi^u(y_{[-n,n]})$ only depends on $y_{-n}\ld y_n$. Further,
(\ref{e.fibres_via_xi_u}) follows directly from the existence and uniqueness of the sequences
$\xi^u(y)$ -- each element $x\in \eta_h^{-1}(y)$ coincides with $\xi^u(y)$ for $u=x_{[0,\ell-2]}$.

In order to see that different elements of the same fibre are $(\ell-1)$-separated, note that similar
to above, for each $n\in\N$ there exists a unique sequence $\xi^{u,n}(y)\in\cA_\kappa^\Z$ such that
$\eta_h\lb\xi^{u,n}(y)\rb=y$ and $\xi^{u,n}(y)_{[n,n+\ell-2]}=u$ (for $n=0$, we have
$\xi^{u,0}(y)=\xi^u(y)$).  Now, if $x, x'\in\eta^{-1}_h(y)$ agree on some interval $[n,n+\ell-2]$ with
$n\in\Z$, say $x_{[n,n+\ell-2]}=x'_{[n,n+\ell-2]}=u$, this yields $x = \xi^{u,n}(y) = x'$.  Hence,
different elements of the same fibre can never coincide on any interval of \mbox{length $\ell-1$.}
\qed\medskip

\begin{remark} \label{r.inverse_branches_general_notation}
  \begin{itemize}
  \item[(a)] In \cite{Hedlund1969SlindingBlockCodes} the question of the separation of fibre elements is
    considered in the more general setting of surjective sliding block codes, where proofs are more
    involved. The fact that different fibre elements of a permutative sliding block codes are always
    $(\ell-1)$-separated follows from \cite[Theorem 11.5]{Hedlund1969SlindingBlockCodes}, in
    combination with the constant cardinality of the fibres. 
  \item[(b)] We note that with the notation of the previous proof, we have
  \[
    \xi^{u,n}(y) \ = \ \begin{cases} \xi^{\Psi_{y_{[0,n-1]}}(u)}(y) & \textrm{ if }n\geq 0\\
                         \xi^{\Phi_{y_{[n,-1]}}(u)}(y) & \textrm{ if } n<0
                       \end{cases}
  \]
\end{itemize}
\end{remark}

We can now establish the skew product representation stated in
Theorem~\ref{t:skew_product_representation}.

\proof[\textit{\bfseries Proof of Theorem~\ref{t:skew_product_representation}}] Suppose that
$h\in \cF(\kappa,\ell)$ is a permutative block code, $Y\ssq \cA_\kappa^\Z$ is a subshift and
$X=\eta_h^{-1}(Y)$. Let
  \begin{eqnarray}\label{e.conjugacy_to_skew_product_2}
     \gamma : X \to Y\times\cA_\kappa^{\ell-1} & , & x \mapsto \left(\eta_h(x),x_{[0,\ell-2]}\right) \ , 
  \\ \label{e.skew_product_system_2}
      T : Y\times \cA_\kappa^{\ell-1}\to Y\times\cA_\kappa^{\ell-1} & , & (y,u)\mapsto
      \left(\sigma(y),\Phi_{y_0}(u)\right) \ .
  \end{eqnarray}
  We need to show that both $\gamma$ and $T$ are homeomorphisms and $\gamma$ is a conjugacy between
  the systems $(X,\sigma)$ and $(Y\times\cA_\kappa^{\ell-1},T)$.

  The continuity of both mappings is obvious from the definitions. Due to
  Proposition~\ref{p.intro_inverse_branches}, the inverse of the map $\gamma$ is explicitly given by
\[
    \gamma^{-1} : Y\times\cA_\kappa^{\ell-1}\to X \quad , \quad (y,u)\mapsto \xi^u(y) \ .
    \]
    The map $T$ is bijective since both $\sigma$ and the mappings $\Phi_{y_0}$ with $y_0\in\cA_\kappa$
    are bijective.  Hence, both $\gamma$ and $T$ are continuous bijections between compact Hausdorff
    spaces, and therefore homeomorphisms. The fact that $\gamma$ is indeed a conjugacy follows from a
    direct computation:
	\begin{eqnarray*}
	\gamma(\sigma(x)) & = & \left(\eta_h\circ \sigma(x), \sigma(x)_{[0,n-2]}\right)\ =
        \ \left(\sigma\circ \eta_h (x), \Phi_{\eta_h(x)_0}(x_{[0,n-2]})\right)\\ & = &
        T\left(\eta_h(x), x_{[0,n-2]}\right)\ = \ T\left(\gamma(x)\right).
	\end{eqnarray*}
\qed\medskip

For later use, we also include the following

\begin{lem}\label{l:independence_lemma}
	Consider a permutative block code $h\in \cF(\cA, \ell)$. Then for any $w,v\in \cA_\kappa^*$
        and $u\in \cA_\kappa^{\ell-1}$, the mapping $\cA_\kappa^{\ell-1}\rightarrow
        \cA_\kappa^{\ell-1},\ u'\mapsto \Phi_{wu'v}(u)$ is a bijection.
\end{lem}
Note that we include the cases where $w$ or $v$ equals the empty word $e$.
\begin{proof}
	Fix $w,v\in\cA^*_\kappa$ and $u\in \cA_\kappa^{\ell-1}$. Let $\bar
        u\in\cA_\kappa^{\ell-1}$. Then the word $\xi^u(w)\zeta^{\bar u}(v)$, which starts with $u$ and
        ends in $\bar u$, is mapped by $h_{|w|+|v|+\ell-1}$ to a word $wu'v$, where
        $u'\in\cA_\kappa^{\ell-1}$ is given by $u'=h_{\ell-1}\lb\Phi_w(u)\Psi_v(\bar u)\rb$. Hence, we
        have that
        \[
        \xi^u(wu'v) \ = \ \xi^u(w)\zeta^{\bar u}(v)
        \]
        ends with $\bar u$, so that $\Phi_{wu'v}(u)=\bar u$ by definition of the induced prefix
        permutation $\Phi_{wu'v}$. This proves the surjectivity of the mapping $u'\mapsto
        \Phi_{wu'v}(u)$, which must therefore be a permutation of $\cA_\kappa^{\ell-1}$ as claimed.
\end{proof}

%%%%%%%%%%%%%%%%%%%%%%%%%%%%%%%%%%%%%%%%%%%%%%%%%%%%%%%%%%%%%%%%%%%%%%%%%%%%%%%%%%%%%%%%
%%%%%%%%%%%%%%%%%%%%%%%%%%%%%%%%%%%%%%%%%%%%%%%%%%%%%%%%%%%%%%%%%%%%%%%%%%%%%%%%%%%%%%%%
%%%%%%%%%%%%%%%%%%%%%%%%%%%%%%%%%%%%%%%%%%%%%%%%%%%%%%%%%%%%%%%%%%%%%%%%%%%%%%%%%%%%%%%%

\section{Minimality of lifts under permutative sliding block codes}\label{Minimality}

\subsection{Minimality via return permutations}\label{Minimality_via_return_permutations}

The main goal of this section is to provide a proof of Theorem~\ref{t:intro_minimality}. However, we
start with the observation that the lift of a minimal subshift can always be decomposed into a finite
number of minimal components. This is a consequence of the following

\begin{thm}[{\cite[Theorems 12.4]{Hedlund1969SlindingBlockCodes}}]\label{t:preimage_of_ap_is_ap}
	Let $h\in \cF(\kappa,\ell)$ be a permutative block code.  If $y\in \cA_\kappa^{\Z}$ is
        almost periodic, then any $x\in \eta^{-1}_h(y)$ is almost periodic.
\end{thm}

\begin{cor}\label{c.number_of_min_components}
  Let $h\in \cF(\kappa,\ell)$ be a permutative block code and $Y\ssq \cA_\kappa^\Z$ a minimal subshift. Then
  $X=\eta_h^{-1}(Y)$ is the disjoint union of at most $\kappa^{\ell-1}$ minimal sets.
\end{cor}
\proof\quad Due to Theorem~\ref{t:preimage_of_ap_is_ap}, every point in $X$ is almost periodic and
therefore contained in some minimal set. Consequently, $X$ is a disjoint union of minimal sets. Since
$Y$ is minimal, we must have $\eta_h(M)=Y$ for every shift-invariant closed subset $M\ssq
X$. Therefore every minimal set must intersect each fibre $\eta_h^{-1}(y)$, $y\in Y$. However, as the
fibres contain only $\kappa^{\ell-1}$ points by Proposition~\ref{p.intro_inverse_branches}, there can
also be at most $\kappa^{\ell-1}$ minimal sets. \qed\medskip

Note that both these statements hold in the more general situation that $\eta_h$ is surjective (but
$h$ is not necessarily permutative), see \cite{Hedlund1969SlindingBlockCodes}.\medskip We can now turn
to the

        \proof[\textit{\bfseries Proof of Theorem \ref{t:intro_minimality}}]

        We note that (ii)\follows(iv) is obvious and show the implications (i)$\follows$(ii),
        (iv)\follows(iii) and (iii)$\follows$(i). \smallskip

        (i)\follows(ii).\quad Let $w\in\cW(Y)$ and $u,u'\in\cA_\kappa^{\ell-1}$. Then both $\xi^u(w)$
        and $\xi^{u'}(w)$ are contained in
       \[
       \cW(X)\ = \ \bigcup_{k=0}^{\ell-1}\cA_\kappa^k \cup \bigcup_{u\in\cA_\kappa^{\ell-1}}\xi^u(\cW(Y)) \ .
       \]
       Hence, by minimality, every $x=\xi^u(y)\in X$ contains both $\xi^u(w)$ and $\xi^{u'}(w)$
       infinitely often as subwords. By shifting if necessary, we may assume
       $x_{[0,|w|+\ell-2]}=\xi^u(w)$, which implies $y_{[0,|w|-1]}=w$.  Further, if
       $\xi^u(y)_{[k,k+|w|+\ell-2]}=\xi^{u'}(w)$, then we must also have $y_{[k,k+|w|-1]}=w$ and
       $\xi^{u}(y)_{[k,k+\ell-2]}=u'$. Consequently, we obtain $\Phi_{wv}(u)=u'$, where
       $v=y_{[|w|,k]}$. Since $u,u'\in\cA_\kappa^{\ell-1}$ were arbitrary, this shows that $\cR(w)$
       acts transitively on $\cA_\kappa^{\ell-1}$.\smallskip
         
       (iv)\follows(iii).\quad Suppose that (iv) holds, fix $n\in\N$ and let $k\geq n$. Let
       $\Phi_{w^{(k)}v}$ be a return permutation of $w^{(k)}$, that is,
       $w^{(k)}vw^{(k)}\in\cW(Y)$. Since $w^{(k)}$ starts with $w^{(n)}$, we have $w^{(k)}=w^{(n)}w'$
       for some $w'\in\cA_\kappa^*$. This means that $w^{(n)}w'vw^{(n)}\in\cW(Y)$, so that
       $\Phi_{w^{(k)}v}=\Phi_{w^{(n)}w'v} \in\cR(w^{(n)})$. As $k\geq n$ and
       $\Phi_{w^{(k)}v}\in\cR(w^{(k)})$ were arbitrary, this shows that
       $\bigcup_{k\geq n}\cR\lb w^{(k)}\rb\ssq \cR\lb w^{(n)}\rb$. Therefore $\cR\lb w^{(n)}\rb$ acts
       transitively on $\cA_\kappa^{\ell-1}$, as required in (iii).  \smallskip
         
       (iii)\follows(i).\quad Suppose for a contradiction that (iii) holds, but there exist a minimal
       strict subset $M\subsetneq X$. Let $M'=X\smin M$ and note that $M'$ is compact as a finite
       union of minimal sets by Corollary~\ref{c.number_of_min_components}. Then, since $M$ and $M'$
       must have a positive distance to each other, there exists $n\in \N$ such that
         \[
         \cW_{|w^{(n)}|+\ell-1}(M)\cap \cW_{|w^{(n)}|+\ell-1}(M') \ = \ \varnothing \ .
         \]
         In order to simplify notation, we let $w=w^{(n)}$ and $m=|w^{(n)}|+\ell-1$ for the remainder
         of this proof.  Now, let $u,u'\in\cA_\kappa^{\ell-1}$ be such that $\xi^{u}(w)\in \cW_{m}(M)$
         and $\xi^{u'}(w)\in \cW_{m}(M')$. Note that since $M$ and $M'$ intersect every fibre due to
         the minimality of $(Y,\sigma)$, such $u,u'$ must always exist.  Using the fact that $\cR(w)$
         acts transitively on $\cA_\kappa^{\ell-1}$, we can choose a finite sequence
         $u^{(0)}=u,u^{(1)}\ld u^{(q)}=u'$ in $\cA_\kappa^{\ell-1}$ and a corresponding sequence of
         return permutations $\Phi_{wv^{(i)}}\in\cR(w)$, $i=0\ld q-1$, such that
         \[
             \Phi_{wv^{(i)}}\lb u^{(i)}\rb \ = \ u^{(i+1)}
       \] 
       holds for $i=0\ld q-1$. We now show, by finite induction on $i$, that
       \begin{equation} \label{e.min_criterion_induction}
                    \xi^{u^{(i)}}(w) \in \cW_{m}(M)
       \end{equation}
       for $i=0\ld q$, contradicting the assumption that $\xi^{u^{(q)}}(w)=\xi^{u'}(w)\in\cW_{m}(M')$
       and $\cW_m(M)\cap\cW_m(M')=\emptyset$.  
       
       For $i=0$, condition (\ref{e.min_criterion_induction}) holds by choice of $u=u^{(0)}$. For the
       induction step $i\to i+1$, choose $y\in\cA_\kappa^\Z$ such that
       $y_{[0, |wv^{(i)}w|-1]}=wv^{(i)}w$. Then, since $\xi^{u^{(i)}}(w)$ is contained in $\cW_m(M)$
       (and therefore not in $\cW_m(M')$), we must have $\xi^{u^{(i)}}(y)\in M$. However, we also have
       $\Phi_{wv^{(i)}}\lb u^{(i)}\rb=u^{(i+1)}$, which implies
       $\xi^{u^{(i)}}(y)_{[|wv^{(i)}|,|wv^{(i)}|+\ell-2]}=u^{(i+1)}$ and therefore
       \[
       \xi^{u^{(i)}}(y)_{[|wv^{(i)}|,|wv^{(i)}|+|w|+\ell-2]}\ = \ \xi^{u^{(i+1)}}(w) \ .
       \]
      Hence, the word $\xi^{u^{(i+1)}}(w)$ appears in the sequence $\xi^{u^{(i)}}(y)\in M$ and thus
      lies in $\cW_m(M)$, as required. This completes the proof.
     \qed

%%%%%%%%%%%%%%%%%%%%%%%%%%%%%%%%%%%%%%%%%%%%%%%%%%%%%%%%%%%%%%%%%%%%%%%%%%%%%%%%%%%%%%%%

\subsection{An example family}\label{ExampleFamily}

As specific examples, we consider the family of block codes $h_{\kappa,\ell}\in\cF(\kappa,\ell)$,
given by
\begin{equation}\label{e.def_example_family}
  h_{\kappa,\ell} : \cA_\kappa^\ell\to\cA_\kappa \quad , \quad w_0\ldots w_{\ell-1} \mapsto
  w_{\ell-1}-w_0\bmod \kappa \ ,
\end{equation}
with integer parameters $\kappa,\ell\geq 2$.  Recall that we write $\eta_{\kappa,\ell}$ instead of
$\eta_{h_{\kappa,\ell}}$.  In the further discussion, it will be convenient to identify the alphabet
$\cA_\kappa$ with the cyclic group $\Z_\kappa=\Z/\kappa\Z$ of order $\kappa$.  Likewise, we consider
$\cA_{\kappa}^n$ as the direct product group $\Z_\kappa^n$, equipped with the coordinate-wise addition
as group structure. In this context, we will sometimes write $(w_0\ld w_{n-1})$ instead of
$w_0\ldots w_{n-1}$ for words in $\cA_\kappa^n$.  Given $w,v\in\cA_\kappa^*$ such that $v=\xi^u(w)$
for some $u\in\cA_\kappa^{\ell-1}$, we see from (\ref{e.def_example_family}) that
$v_{n+\ell-1}=v_n+w_n \bmod \kappa$ and hence, by induction,
\begin{equation}
  v_{n+k(\ell-1)} \ = \ v_n + \sum_{j=0}^{k-1} w_{n+j(\ell-1)}\bmod \kappa \ .
\end{equation}
This allows for a straightforward computation of the induced prefix permutations. For $k=0\ld \ell-2$,
we let
\begin{equation}
  t^k(w) \ = \ \sum_{j=0}^{\lfloor (|w|-k-1)/(\ell-1)\rfloor} w_{k+j(\ell-1)} \ \bmod \kappa \ ,
\end{equation}
where $\lfloor\alpha\rfloor=\sup\{n\in\Z\mid n\leq \alpha\}$ for all $\alpha\in\R$. In other words,
$t^k(w)$ is the sum of the entries of $w$ over all positions in the set
$\lb k+(\ell-1)\Z\rb \cap [0,|w|-1]$.  Then, if $|w|$ is a multiple of $\ell-1$ (which is the simpler
case) we have
\begin{equation} \label{e.return_permutation_formula_1} \Phi_w(u_0\ld u_{\ell-2}) \ =
\ (u_{0}\ld u_{\ell-2}) + (t^0(w)\ld t^{\ell-2}(w)) \ ,
\end{equation}
(with addition carried out in $\cA_\kappa^{\ell-1}\simeq \Z_\kappa^{\ell-1}$). In general, we have
\begin{equation}\label{e.return_permutation_formula_2} 
  \begin{split} & \Phi_w(u_0\ld u_{\ell-2})  \ = \ (u_{|w|\bmod (\ell-1)}\ld u_{\ell-2},u_{0}\ld u_{|w|-1\bmod(\ell-1)})\\
   & \qquad + (t^{|w|\bmod (\ell-1)}(w)\ld t^{\ell-2}(w),t^{0}(w)\ld t^{|w|-1\bmod(\ell-1)}(w)) \ .
  \end{split}
\end{equation}
\begin{remark}\label{r.transitive_permutation_pair}
  For the application of Theorem~\ref{t:intro_minimality}, we need to ensure that the family of return
  permutations for a given word $w$ acts transitively on $\cA_\kappa^{\ell-1}$. This is certainly true
  if it contains the two permutations
  \[
  \begin{split}
  & \pi : (u_0,u_1\ld u_{\ell-2})\mapsto (u_1\ld u_{\ell-2},u_0) \eqand\\ &  \pi': (u_0,u_1\ld u_{\ell-2})\mapsto
  (u_0\ld u_{k-1},u_k+1,u_{k+1}\ld u_{\ell-2}) 
  \end{split}
  \]
  for some $k\in\{0\ld \ell-2\}$.  When $h=h_{\kappa,\ell}$ for some $\kappa,\ell\geq 2$, so that the
  return permutations are given by (\ref{e.return_permutation_formula_1}) and
  (\ref{e.return_permutation_formula_2}), this is always the case if there exist $v,v'\in\cA_\kappa^*$
  such that
  \begin{itemize}
    \item $wvw,wv'w\in\cW(Y)$,
    \item $|wv|\bmod \ell-1=1$;
    \item $(t^0(wv),t^1(wv)\ld t^{\ell-2}(wv))=(0,0\ld 0)$;
    \item $|wv'| \bmod \ell-1=0$;
    \item $\exists! k\in\{0\ld \ell-2\}: t^k(wv')=1$ and $t^j(wv')=0$ for all $j\neq k$.
  \end{itemize}
  In this case, we have $\Phi_{wv}=\pi$ and $\Phi_{wv'}=\pi'$. This will be instrumental for the
  application to substitution subshifts in Section~\ref{SubstitutionApplications} below.
\end{remark}
In the special case $\ell=2$, all induced prefix permutations are simply rotations on the cyclic group
$A_\kappa\simeq \Z_\kappa$. More precisely, we have
\[
     \Phi_w(a) \ = \ a+\sum_{j=0}^{|w|-1}w_j \ \bmod \kappa \ .
     \]
     This means, in particular, that all induced prefix permutations commute with each
     other. Moreover, in this case the skew product map $T$ in
     Theorem~\ref{t:skew_product_representation} is a group extension, since all fibre maps
     $\Phi_{y_0}:a\mapsto a+y_0\bmod \kappa$ are rotations on the cyclic group $\cA_\kappa$.  If
     $w\in\cA_\kappa^*$ contains a subword $\tilde w\triangleleft w$, say $w=w'\tilde{w}w''$, and
     $wvw\in\cW(Y)$, then we also have $\tilde w(w''vw')\tilde w\in\cW(Y)$ (since this is a subword of
     $wvw$) and (\ref{e.induced_permutation_composition}) together with the commutativity yields
\[
\Phi_{\tilde w(w''vw')} \ = \ \Phi_{w'}\circ \Phi_{\tilde w(w''vw')} \circ \Phi_{w'}^{-1} \ = \ \Phi_{wv} \ .
\]
As a consequence, we obtain $\cR(w)\ssq \cR(\tilde w)$. Further, in the situation of
Theorem~\ref{t:intro_minimality}(iii), the minimality of $Y$ together with the fact that
$|w^{(n)}|\nKonv \infty$ imply that for each $n\in\N$ there exists $N\geq n$ such that
$w^{(n)}\triangleleft w^{(k)}$ holds for all $k\geq N$. This leads to the following simplified
criterion for the minimality of lifts under the sliding block codes $\eta_{2,\kappa}$, $\kappa\in\N$.
\begin{prop}\label{p.h2lift_substitution_minimality}
  Let $Y\ssq \cA_\kappa^\Z$ be a minimal subshift, $\kappa\geq 2$ and $X=\eta_{2,\kappa}^{-1}(Y)$.
  Then $X$ is minimal if and only if there exists a sequence of words $w^{(n)},v^{(n)}\in\cW(Y)$ such
  that
  \begin{itemize}
  \item[(i)] $w^{(n)}v^{(n)}w^{(n)}\in\cW(Y)$ for all $n\in\N$;
  \item[(ii)] $|w^{(n)}|\nKonv \infty$;
  \item[(iii)] the sequence $(t_n)_{n\in\N}$ in $\Z_\kappa$ given by
  \begin{equation}\label{e.h2_substitution_minimality_def_tn}
  t_n \ = \ t^0\left(w^{(n)}v^{(n)}\right) \ = \ \sum_{k=0}^{|w^{(n)}v^{(n)}|-1}
  \left(w^{(n)}v^{(n)}\right)_k \bmod \kappa \
  \end{equation}
  has the property that $\{t_k\mid k\geq N\}$ generates $\Z_\kappa$ for all $N\in\N$.
  \end{itemize}
\end{prop}
  \begin{proof}
    Suppose that (i)--(iii) hold and fix $n\in\N$. By minimality, we can choose $N\geq n$ such that
    $w^{(n)}\triangleleft w^{(k)}$ for all $k\geq N$. Then, by the above, we have
    $\cR(w^{(k)})\ssq \cR(w^{(n)})$ for all $k\geq N$. Therefore $\cR(w^{(N)})$ contains all the
    permutations
    \[
    \Phi_{v^{(k)}w^{(k)}} : \cA_\kappa\to\cA_\kappa \quad , \quad a \mapsto a+t_k \bmod \kappa \ 
    \]
    with $k\geq N$. Since $\{t_k\mid k\geq N\}$ generates $\Z_k$, this implies that $\cR(w^{(n)})$
    acts transitively on $\cA_\kappa$. Hence, Theorem~\ref{t:intro_minimality}(iii) yields the
    minimality of $X$.\smallskip

    Conversely, suppose that $X$ is minimal. Choose some $y\in Y$ and let $w^{(n)}=y_{[0,n]}$. Then by
    minimality $\xi^1\lb w^{(n)}\rb$ occurs in $\xi^0(y)$, say
    $\xi^1\lb w^{(n)}\rb = \xi^0(y)_{[k,k+n]}$, where we can assume without loss of generality that
    $k>n$. If we let $v^{(n)}=y_{[n+1,k-1]}$, so that $w^{(n)}v^{(n)}=y_{[0,k-1]}$, then we obtain
    that $y_{[0,k+n]}=w^{(n)}v^{(n)}w^{(n)}$ and
    \[
      \Phi_{w^{(n)}v^{(n)}}(0) \ = \ \Phi_{y_{[0,k-1]}}(0) \ = \ 1 \ ,
    \]
    which further implies $t_n=t^0\lb w^{(n)}v^{(n)}\rb=1$. Hence, properties (i)--(iii) are
    satisfied.
  \end{proof}

  Note that in the special case $\kappa=2$, Corollary~\ref{c.intro_minimality_criterion} follows as a
  direct consequence.

  %%%%%%%%%%%%%%%%%%%%%%%%%%%%%%%%%%%%%%%%%%%%%%%%%%%%%%%%%%%%%%%%%%%%%%%%%%%%%%%%%%%%%%%%
  
  \subsection{Applications to substitution systems} \label{SubstitutionApplications}

  The aim of this section is to illustrate the above minimality criteria by the application to some
  well-known substitution systems. More precisely, we will have a look at the noble means substitution
  with parameter $p\in\N$ ($\mathrm{nm},p$) (which includes the Fibonacci and the silver mean
  substitution), the Tribonacci substitution ($\mathrm{Tri}$) and the period doubling substitution
  ($\mathrm{pd}$). It will become clear from the proofs below that there is considerable potential for
  further generalisations. However, our intention here is not to perform a systematic study of the
  largest-possible classes of substitutions to which our methods can be applied -- this will be
  subject of subsequent work \cite{GohlkeJaegerKang2026PermutativeSubstitutionLifts}. Instead, the
  results presented below should rather be understood as a `proof of principle'.

  In order to avoid redundancies, we fix some uniform notation for all examples. First, the respective
  substitution rule will be denoted by $\Gamma_\alpha$, where
  $\alpha\in\{(\mathrm{nm},p),\mathrm{Tri},\mathrm{pd}\}$. The $n$-th substitution word is defined as
  \begin{equation}\label{e.def_substitution_words}
    \tau^{(n)}\ = \ \Gamma^n_\alpha(0) \ ,
  \end{equation}
  with the exception of period doubling, where we let $\tau^{(n)}=\Gamma_\mathrm{pd}^n(1)$ (see
  Section~\ref{PeriodDoubling} for an explanation). Note that we keep the dependence of the
  substitution words on $\alpha$ implicit.  Each of these substitutions induces a subshift $X_\alpha$,
  which is defined via (\ref{e.permitted_words_subshift}) as $X_\alpha = X_{W(\alpha)}$, where
  $W(\alpha)=\left\{w\in\cA^*_\kappa\mid \exists n\in\N: w\triangleleft \tau^{(n)}\right\}$.  By
  \begin{equation}
    \label{e.def_substitution_word_length} r_n \ = \ |\tau^{(n)}| \ \bmod (\ell-1)\  , 
  \end{equation}
  we denote the length of the $n$-th substitution word $\bmod \ell-1$. Together with the coefficients
  \begin{equation}\label{e.def_return_permutation_parameters}
    t^k_n \ = \ t^k\lb \tau^{(n)}\rb \ = \ \sum_{j=0}^{\lfloor (|w^{(n)}|-k-1)/(\ell-1)\rfloor}
    \tau^{(n)}_{k+j(\ell-1)} \ \bmod \kappa \ ,
  \end{equation}
  for $n\in\N$ and $k=0\ld \ell-2$, these parameters determine the prefix permutations
  $\Phi_{\tau^{(n)}}$, as described by (\ref{e.return_permutation_formula_1}) and
  (\ref{e.return_permutation_formula_2}) above. When $\ell=2$, we also write
  \begin{equation} \label{e.def_substitution_word_index_sum}
    t_n \ = \ \sum_{j=0}^{|\tau^{(n)}|-1}\tau^{(n)}_j 
  \end{equation}
  instead of $t_n^0$.  We refer to \cite{BaakeGrimm2013AperiodicOrder,Bruin2022SymbolicDynamics} for a
  general background on substitution systems, as well as for further specific information about the
  examples discussed below.

%%%%%%%%%%%%%%%%%%%%%%%%%%%%%%%%%%%%%%%%%%%%%%%%%%%%%%%%%%%%%%%%%%%%%%%%%%%%%%%%%%%%%%%%

\subsubsection{Noble means substitutions} \label{NobleMeans}

The {\em noble means substitution} with parameter $p\in\N$ is given by the substitution rule
\begin{equation}
\Gamma_{\mathrm{nm},p} : \cA_2\to\cA_2^2 \quad , \quad \left\{
\begin{split}
  0 & \mapsto & 0^p1\\
  1 & \mapsto & 0
\end{split} \right\} \ .
\end{equation}
Note that for $p=1$, we obtain the Fibonacci substitution $0\mapsto 01,\ 1\mapsto 0$. Moreover,
$\Gamma_{\mathrm{nm},2}$ is conjugate to the {\em silver mean substitution}
$ \Gamma_{\mathrm{silver}} : \cA_2\to\cA_2^2, \ 0 \mapsto 010,\ 1 \mapsto 0$ by the concatenation with
$0$ from the left, respectively the right, in the sense that
$\Gamma^n_{\mathrm{nm},2}(0)0=0\Gamma^n_\mathrm{silver}(0)$. Therefore, if $X_\mathrm{silver}$ denotes
the subshift induced by the silver means substitution, then $X_{\mathrm{nm},2}=X_\mathrm{silver}$ (see
\cite[Section 4.4]{BaakeGrimm2013AperiodicOrder} for details). Consequently, we do not have to
distinguish between these examples and only treat the general case of the noble means substitutions.
The substitution words satisfy the recursive equation
\begin{equation}\label{e.noble_means_recursive_equation}
  \tau^{(n+1)} \ = \ \left(\tau^{(n)}\right)^p\tau^{(n-1)} \ .
\end{equation}

\begin{thm}\label{t.noble_means_minimality}
  For all parameters $p,\kappa,\ell\geq 2$, the subshift $X=\eta^{-1}_{\kappa,\ell}\lb
  X_{\mathrm{nm},p}\rb$ is minimal.
\end{thm}
\proof\quad In order to apply Theorem~\ref{t:intro_minimality}(iv), we let $w^{(n)}=\tau^{(n-1)}$
and $v^{(n)}=\lb\tau^{(n-1)}\rb^{p-2}\tau^{(n-2)}$, so that we obtain
$w^{(n)}v^{(n)}=\tau^{(n)}$. Note that in addition we have
$w^{(n)}v^{(n)}w^{(n)}=\tau^{(n)}\tau^{(n-1)}\in\cW\lb X_{\mathrm{nm},p}\rb$, which means that
$\Phi_{\tau^{(n)}}\in\cR\lb w^{(n)}\rb$.

For the length parameters $r_n$, (\ref{e.noble_means_recursive_equation}) yields $r_{n+1}=p\cdot
r_n+r_{n-1}\bmod \ell-1$. Consequently, we have
\[
\twovector{r_n}{r_{n+1}} \ = \ \Matrix{cc}{0&1\\1&p} \cdot \twovector{r_{n-1}}{r_n} \ .
\]
As the matrix $\Matrix{cc}{0&1\\1&p} $ acting on the finite space $\Z_{\ell-1}^2$ is invertible, the
sequence \nfolge{r_n} must be periodic.  For the parameters $t^k_n$ defined by
(\ref{e.def_return_permutation_parameters}), we obtain the recursive formulas
    \begin{equation}\label{e.noblemeans_general_recursion}
  t^k_{n+1} \ = \ \left(\sum_{j=0}^{p-1} t^{k-jr_n\bmod (\ell-1)}_{n}\right) +
  t_{n-1}^{k-pr_n\bmod (\ell-1)} \ \bmod \kappa \ .
  \end{equation}
  If we define vectors $V_n=\lb t^0_{n-1}\ld t^{\ell-2}_n,t^0_n\ld t^{\ell-2}_n\rb^t$, then
  (\ref{e.noblemeans_general_recursion}) translates to
    \[
        V_{n+1} \ = \ A_{r_n} \cdot V_n \ ,
        \]
        where the matrices $A_{r_n}$ acting on $\Z_\kappa^{2(\ell-1)}$ are again invertible. The
        latter comes from the fact that the parameters $t^k_{n-1}$, $k=0\ld \ell-2$, only appear once
        in the equations (\ref{e.noblemeans_general_recursion}), and therefore $V_n$ can be recovered
        from the data contained in $V_{n+1}$. Alternatively, one may also compute the determinants of
        the $A_ {r_n}$ over the ring $\Z_\kappa$, which are non-vanishing.

        As the matrices $A_{r_n}$ form a periodic sequence acting on a finite space, the vectors $V_n$
        must be periodic. This further yields the periodicity of the induced prefix permutations
        $\Phi_{\tau^{(n)}}$ via (\ref{e.return_permutation_formula_2}). Therefore, it suffices to show
        that the family of permutations $\left\{\Phi_{\tau^{(n)}}\mid n\in\N\right\}$ acts
        transitively on $\cA^{\ell-1}$. However, this is already true for the two permutations
   \begin{eqnarray*}
     \Phi_{\tau^{(0)}} = \Phi_0 & : & \quad u_0\ldots u_{\ell-2} \ \mapsto \ u_1\ldots u_{\ell-2}u_0
     \ , \\ \Phi_{\tau^{(1)}} = \Phi_{0^{p-1}1} & : & \quad u_0\ldots u_{\ell-2} \ \mapsto
     \ u_{p\bmod(\ell-1)}\ldots u_{\ell-2}u_0 \ldots \lb u_{p-1\bmod \ell-1}+1\rb \ .
   \end{eqnarray*}
   Consequently, Theorem~\ref{t:intro_minimality} implies the minimality of
   $X=\eta^{-1}\lb X_{\mathrm{nm},p}\rb$ as required.\qed 

%%%%%%%%%%%%%%%%%%%%%%%%%%%%%%%%%%%%%%%%%%%%%%%%%%%%%%%%%%%%%%%%%%%%%%%%%%%%%%%%%%%%%%%%

\subsubsection{Tribonacci substitution}\label{Tribonacci}

The {\em Tribonacci substitution} on three symbols is given by
\begin{equation}
\Gamma_{\mathrm{Tri}} : \cA_3\to\cA_3^2 \quad , \quad \left\{
\begin{split}
  0 & \mapsto & 01\\
  1 & \mapsto & 02\\
  2 & \mapsto & 0 
\end{split} \right\} \ .
\end{equation}
For the substitution words, we have the recursive formula
\begin{equation}\label{e.recursion_tribonacci}
  \tau^{(n+1)} \ = \ \tau^{(n)}\tau^{(n-1)}\tau^{(n-2)} \ . 
\end{equation}
The minimality of the lifts of the {\em Tribonacci subshift} $X_\mathrm{Trib}$ follows by the same
general pattern as in the previous example. The decisive observation is that the substitution word
with the lowest index appears only once in the recursive rule (\ref{e.recursion_tribonacci}), which
ensures the invertibility of the `transfer matrices'. Again, we provide some details in the proof
below.
\begin{thm}\label{t.Tribonacci_minimality}
  For all integers $\kappa\geq 3$ and $\ell\geq 2$, the subshifts $X=\eta^{-1}_{\kappa,\ell}\lb
  X_\mathrm{Tri}\rb$ are minimal.
\end{thm}

\proof\quad We let $w^{(n)}=\tau^{(n-1)}$ and $v^{(n)}=\tau^{(n-2)}\tau^{(n-3)}$, so that we have
$w^{(n)}v^{(n)}=\tau^{(n)}$ and
$w^{(n)}v^{(n)}w^{(n)}=\tau^{(n)}\tau^{(n-1)}\in\cW\lb X_\mathrm{Trib}\rb$. The latter yields
$\Phi_{\tau^{(n)}}\in\cR\lb w^{(n)}\rb$.

The lengths $r_n$ satisfy the recursive equation $r_{n+1}= r_n+r_{n-1}+r_{n-2}\bmod \ell-1$, so that
\[
\Matrix{c}{r_{n-1}\\r_n\\r_{n+1}} \ = \ \Matrix{ccc}{0&1& 0\\0&0& 1\\ 1 & 1 & 1} \cdot
\Matrix{c}{r_{n-2}\\r_{n-1}\\r_n} \ .
\]
Due to the bijectivity of the matrix, the sequence \nfolge{r_n} is periodic. For the parameters
$t^k_n$ given by (\ref{e.def_return_permutation_parameters}), we have the recursive equations
    \begin{equation}\label{e.tribonacci_general_recursion}
  t^k_{n+1} \ = \ t^k_n + t^{k-r_n\bmod (\ell-1)}_{n-1} +
  t_{n-2}^{k-r_n-r_{n-1}\bmod (\ell-1)} \ \bmod \kappa \ .
  \end{equation}
  Hence, the vectors $V_n=\lb t^0_{n-2}\ld t^{\ell-2}_{n-2},t^0_{n-1}\ld t^{\ell-2}_n,t^0_n\ld
  t^{\ell-2}_n\rb^t$ satisfy
    \[
        V_{n+1} \ = \ A_{r_n,r_{n-1}} \cdot V_n \ , 
        \]
        with invertible matrices $A_{r_n,r_{n-1}}$ determined by
        (\ref{e.tribonacci_general_recursion}). This yields the periodicity of $\nfolge{V_n}$, and
        hence of $\nfolge{\Phi_{\tau^{(n)}}}$. Consequently, it suffices again to show that the
        family of permutations $\left\{\Phi_{\tau^{(n)}}\mid n\in\N\right\}$ acts transitively on
        $\cA^{\ell-1}$. This is true since it contains the two permutations
   \begin{eqnarray*}
     \Phi_{\tau^{(0)}} = \Phi_0 & : & \quad u_0\ldots u_{\ell-2} \ \mapsto \ u_1\ldots u_{\ell-2}u_0
     \ , \\ \Phi_{\tau^{(1)}} = \Phi_{01} & : & \quad u_0\ldots u_{\ell-2} \ \mapsto
     \ u_{2}\ldots u_{\ell-2}u_0 \ldots \lb u_{1}+1\rb \ .
   \end{eqnarray*}
   As before, Theorem~\ref{t:intro_minimality} now implies the minimality of $X=\eta^{-1}_{\kappa, \ell}\lb
   X_{\mathrm{Tri}}\rb$.  \qed\medskip

%%%%%%%%%%%%%%%%%%%%%%%%%%%%%%%%%%%%%%%%%%%%%%%%%%%%%%%%%%%%%%%%%%%%%%%%%%%%%

 \subsubsection{Period doubling substitution}
  \label{PeriodDoubling}
  The {\em period-doubling substitution} is given by 
\begin{equation}
\Gamma_{\mathrm{pd}} : \cA_2\to \cA_2^2 \quad, \quad \left\{
\begin{split}
0 & \mapsto &11\\
1 & \mapsto &10
\end{split} \right\}.
\end{equation}
We note that in comparison with the usual representation of the period doubling substitution, we
interchange the roles of the symbols $0$ and $1$. The reason is the fact that in this way, the
Thue-Morse subshift can be obtained as the lift of $X_\mathrm{pd}$ under $\eta_{2,2}$. Otherwise, we
would have to replace $h_{2,2}$ by the block code $\tilde h_{2,2}:x_0x_1\mapsto x_1-x_0+1\bmod 2$ in
order to obtain the same result. If we let
\begin{equation}\label{e.def_substitution_words_pd}
\tau^{(n)}\ = \ \Gamma^n_{\mathrm{pd}}(1) \eqand \vartheta^{(n)}\ = \ \Gamma^n_{\mathrm{pd}}(0)\ , 
\end{equation}
then it is straightforward to check that
\begin{eqnarray}\label{e.period_doubling_recursive_equation}
\tau^{(n+1)} \ = \ \tau^{(n)}\vartheta^{(n)}, \
\vartheta^{(n+1)} \ = \ \tau^{(n)}\tau^{(n)} .
\end{eqnarray}
Alternatively, we can also represent $\tau^{(n+1)}$ as
\begin{equation}
  \label{e.period_doubling_recursive_equation_2}
  \tau^{(n+1)} \ = \ \tau^{(n)}\tau^{(n-1)}\tau^{(n-1)} \ .
\end{equation}
Moreover, we have $|\tau^{(n)}| = |\vartheta^{(n)}| = 2^n$ for every $n\in \N$.  We define a mapping
\begin{equation}
\breve \cdot : \cA_{\kappa}^{*}\to \cA_{\kappa}^{*} \quad, \quad w\mapsto \breve w = w_{[0, |w|-2]}
\ ,
\end{equation}
thus reducing the length of each finite word by omitting the last symbol.  Using the recursive
relations \eqref{e.period_doubling_recursive_equation}, a straightforward induction shows that, for
every $n\geq 1$,
\begin{equation}
\breve{\tau}^{(n)} \ = \ \breve{\vartheta}^{(n)},
\end{equation}
while
\begin{equation}
\tau^{(n)}_{2^n-1}\ = \ n + 1 \bmod 2 \ = \ 1 - \vartheta^{(n)}_{2^n-1}\ .
\end{equation}

In order to prove Theorem~\ref{t:intro_period_doubling}, we distinguish four cases ((1) $\ell$ even;
(2) $\ell=3,\ \kappa$ odd; (3) $\ell\geq 5$ odd; (4) $\ell=3,\ \kappa$ even) and treat each of them
separately.

\begin{prop} \label{p.period_doubling_case1}
  For any $\kappa\geq 2$ and any even $\ell\geq 2$, the lift
  $X = \eta_{\ell, \kappa}^{-1}(X_{\mathrm{pd}})$ of the period doubling subshift under
  $\eta_{\ell, \kappa}$ is minimal.
\end{prop}
\begin{proof}
Choose a strictly increasing sequence $(k_n)_{n\in \N}$ of integers with $2^{k_0} > \ell-1$. We shall
apply Theorem~\ref{t:intro_minimality}(iii) with $w^{(n)} = \breve{\tau}^{(k_n)}$.

For each $n\in \N$, choose $m_n>k_n$ such that $2^{m_n} > (\ell-1)\cdot 2^{k_n}$. Since
$\tau^{(n)}_{2^n}$ alternates between $0$ and $1$, we may assume that 
\begin{equation}
  \label{e.period_doubling_even_assumption}
  \tau^{(m_n)}_{2^{m_n}-1} \ = \ 0 \eqand \vartheta^{(m_n)}_{2^{m_n}-1} \ = \ 1
\end{equation}
by increasing $m_n$ if necessary. Recall that by the recursive relations in
(\ref{e.period_doubling_recursive_equation}), we have 
\[
	\tau^{(m_n + 2)} \ = \ \tau^{(m_n)}\vartheta^{(m_n)}\tau^{(m_n)}\tau^{(m_n)}\ .
\]
We may further decompose $\tau^{(m_n+2)}$ into words of length $2^{k_n}$
\begin{equation}\label{e.decomposition_of_w^{(m_n+2)}}
  \begin{split}
	\tau^{(m_n+2)} \ = \ &  \underbrace{v^{(0)}\ldots v^{(2^{m_n - k_n}-1)}}_{\tau^{(m_n)}}
        \underbrace{v^{(2^{m_n - k_n})}\ldots v^{(2\cdot 2^{m_n-k_n}-1)}}_{\vartheta^{(m_n)}} \\ & 
        \underbrace{v^{(2\cdot2^{m_n - k_n})}\ldots v^{(3\cdot 2^{m_n - k_n}-1)}}_{\tau^{(m_n)}}
        \underbrace{v^{(3\cdot 2^{m_n - k_n})}\ldots v^{(4\cdot 2^{m_n - k_n})-1}}_{\tau^{(m_n)}} \ , 
  \end{split}
\end{equation}
where $v^{(i)} \in \left\{\tau^{(k_n)},\ \vartheta^{(k_n)}\right\}$ for $i\in \left\{0\ld 4\cdot
  2^{m_n - k_n}-1\right\}$.  Now fix $i\in {0,\ldots,\ell-2}$ and consider the subwords $v$ and $v'$
of $\tau^{(m_n+2)}$ of length $(\ell-1)\cdot 2^{k_n}$ given by 
\begin{eqnarray*}
 v  & = &  v^{(2^{m_n - k_n}-i-1)}\ldots v^{(2^{m_n - k_n}-i+\ell-3)}  \ , \\
 v' & = &  v^{(2\cdot2^{m_n - k_n}-i-1)}\ldots v^{(2\cdot2^{m_n - k_n}-i+\ell-3)}\ .
\end{eqnarray*}
Note that $v$ and $v'$ are exactly at distance $2^{m_n-k_n}$ of each other and
$v^{(2^{m_n-k_n}-1)}\in \left\{\tau^{(k_n)}, \vartheta^{(k_n)}\right\}$. Further, since
$\breve{\tau}^{(m_n)} = \breve{\vartheta}^{(m_n)}$, we have $v_j = v'_j$ if
$j\neq (i+1)\cdot 2^{k_n}-1$. Due to (\ref{e.period_doubling_even_assumption}), we have $m_n$,
$v_{(i+1)\cdot 2^{k_n}-1}= \tau^{(m_n)}_{2^{m_n}-1} = 0$ and
$v'_{(i+1)\cdot 2^{k_n}-1} = \vartheta^{(m_n)}_{2^{m_n}-1} =1$.  Moreover, it follows from the
decomposition \eqref{e.decomposition_of_w^{(m_n+2)}} that
$\Phi_v, \Phi_{v'}\in \cR\lb \breve{\tau}^{(k_n)}\rb$.  Since $\ell-1$ divides $|v|$,
formula~\eqref{e.return_permutation_formula_1} yields
\[
	\Phi_v(u_0\ld u_{\ell-2})  \ = \ (u_{0}\ld u_{\ell-2}) + (t^{0}(v)\ld  t^{\ell-2}(v)) \ .
\]
where the addition is taken coordinatewise modulo $\kappa$.  Since $|v| = |v'|$, the same formula
holds with $v$ replaced by $v'$.  As $\ell-1$ is odd, we have $\gcd(2^{k_n}, \ell-1)=1$. Hence, the
map
\[
  \alpha: \Z_{\ell-1}\to \Z_{\ell-1} \quad ,\quad j\mapsto (j+1)2^{k_n}-1 \bmod{\ell-1}
\]
is a bijection. From the observations above, it follows that $t^k(v) = t^k(v')$ whenever $k\neq
\alpha(i)$, while $t^{\alpha(i)}(v') = t^{\alpha(i)} + 1 \bmod \kappa$. Therefore, for every $u\in
\cA_{\kappa}^{\ell-1}$ and $k\neq \alpha(i)$.
\[
	\left(\Phi_v^{\kappa-1}\circ \Phi_{v'}(u)\right)_k \ = \ u_k + \kappa\cdot t^k(v) \bmod \kappa
        \ = \ u_k\ ,
\]
while at the same time
\[
\left(\Phi_v^{\kappa-1}\circ \Phi_{v'}(u)\right)_{\alpha(i)} \ = \ u_{\alpha(i)} + \kappa\cdot
t^{\alpha(i)}(v) + 1 \bmod \kappa \ = \ u_{\alpha(i)} +1 \bmod \kappa\ .
\]
As $\Phi_v,\Phi_{v'}\in\cR\lb \breve\tau^{(k_n)}\rb$ and $i\in \{0\ld \ell-2\}$ was arbitrary, this
shows that $\cR(\breve{\tau}^{(k_n)})$ acts transitively on $\cA_{\kappa}^{\ell-1}$ and thus
completes the proof.
\end{proof}

Next, we consider the case $\ell=3$ and $\kappa$ odd.

\begin{prop}
  Suppose that $\kappa\geq 2$ is odd. Then $X=\eta^{-1}_{\kappa,3}\lb X_\mathrm{pd}\rb$ is minimal.
\end{prop}
\begin{proof}
  In this case, the proof is similar to that of Theorem~\ref{t.noble_means_minimality} for the noble
  means subshifts. We aim to apply Theorem~\ref{t:intro_minimality}(iv) with
  \[
    w^{(n)}\ = \ \tau^{(n-1)} \eqand v^{(n)} \ = \ \tau^{(n-2)}\tau^{(n-2)} \ .
  \]
  Note that thus $\Phi_{\tau^{(n)}}\in\cR(w^{(n)})$.  Since $|\tau^{(n)}|=2^n$ for all $n\in\N$ and
  $\ell-1=2$, we have $r_n=|\tau^{(n)}|\bmod (\ell-1)=0$ for all $n\geq 1$. The recursive equation
  (\ref{e.period_doubling_recursive_equation_2}) therefore yields
  \[
  t^0_{n+1} \ = \ t^0_n+2t^0_{n-1} \ \bmod \kappa \eqand t^1_{n+1} \ = \ t^1_n+2t^1_{n-1} \ . 
  \]
  for all $n\geq 2$. If we let $V_n=\lb t^0_n,t^1_n,t^0_{n-1},t^1_{n-1}\rb^t$, then this means
  \[
  V_{n+1} \ = \ A\cdot V_n
  \]
  with
  \[
         A \ = \ \Matrix{cccc}{1 & 0 & 2 & 0\\0 & 1 & 0 & 2\\1 & 0 & 0 & 0\\0 & 1 & 0 & 0} \ . 
         \]
         Now, the matrix $A$ is invertible over the ring $\cA_\kappa=\Z_\kappa$ (note that $\det(A)=4$
         and $\kappa$ is odd). Hence, similar as in the proof of
         Theorem~\ref{t.noble_means_minimality}, we obtain periodicity of the sequence
         $\lb\Phi_{\tau^{(n)}}\rb_{n\geq 1}$ is periodic. (The only difference to the noble means case
         is the fact that the sequence starts at $n=1$ instead of $n=0$.).  Thus, in the light of
         Theorem~\ref{t:intro_minimality}(iv), it remains to show that the family of permutations
         $\{\Phi_{\tau^{(n)}}\mid n\in\N\}$ acts transitively on $\cA^2_{\kappa}$. However, this is
         already true for the pair
      \begin{eqnarray*}
        \Phi_{\tau^{(1)}} & = & \Phi_{10} \ : \ u_0u_1\mapsto (u_0+1)u_1 \ , \\ \Phi^{\tau^{(2)}} & =
        & \Phi_{1011} \ : \ u_0u_1\mapsto (u_0+2)(u_1+1) \ .
      \end{eqnarray*}
      This yields the minimality of $X$.
\end{proof}

It remains to treat the cases where $\eta^{-1}_{\kappa,\ell}\lb X_\mathrm{pd}\rb$ is non-minimal. We
start with odd parameters $\ell\geq 5$. To that end, we introduce two auxiliary mappings
\begin{eqnarray*}
  \pi_\mathrm{even} & : & \cA_\kappa^\Z\to\cA_\kappa^{\N_0} \quad , \quad y \mapsto \nofolge{y_{2n}} \ ,\\
  \pi_\mathrm{odd} & : & \cA_\kappa^\Z\to\cA_\kappa^{\N_0} \quad , \quad y \mapsto \nofolge{y_{2n+1}} \ .
\end{eqnarray*}
\begin{prop} \label{p.period_doubling_non-minimality}
  Suppose that $\kappa\geq 2$ and $\ell\geq 5$ is odd. Then $X=\eta^{-1}_{\kappa,\ell}\lb
  X_\mathrm{pd}\rb$ is non-minimal.
\end{prop}
\begin{proof} Let $y\in X_\mathrm{pd}$ be chosen such that its right side coincides with the period
  doubling sequence, that is, $y_{[0,|\tau^{(n)}|-1]}=\tau^{(n)}$ for all $n\in\N$. We aim to show
  that $\xi^{10\ldots 0}(y)$ cannot be contained in the forward orbit of $\xi^{0\ldots 0}(y)$, thus
  excluding minimality.

  To that end, note that since $y_k=1$ holds for all $n\in 2\N_0$ (a well-known and easy to prove fact
  about the period doubling sequence), we have that
  \begin{eqnarray}\label{e.pd_lift_even_positions}
  \pi_\mathrm{even}\lb\xi^{0\ldots 0}(y)\rb & = & \lb 0^{(\ell-1)/2}1^{(\ell-1)/2} \ldots
  (\kappa-1)^{(\ell-1)/2}\rb^\infty \ , \\
  \label{e.pd_lift_even_postions_2}
    \pi_\mathrm{even}\lb\xi^{10\ldots0}(y)\rb & = & \lb 10^{(\ell-3)/2}21^{(\ell-3)/2}\ldots
    0(\kappa-1)^{(\ell-3)/2}\rb^\infty \ .
  \end{eqnarray}
    Now, let \kfolge{n_k} be any sequence of positive integers such that $n_k\nearrow\infty$ and
    \[
    x \ = \ \kLim \sigma^{n_k}\lb\xi^{0\ldots 0}(y)\rb 
    \]
    exists. By going over to a subsequence if necessary, we may assume that $j=n_k\bmod
    (\ell-1)\kappa$ is constant in $k$. If $j$ is even, then we obtain
    \[
       \pi_\mathrm{even}(x) \ = \ \sigma^{j/2}\lb \pi_\mathrm{even}\lb \xi^{0\ldots 0}(y)\rb\rb \ =
       \ \sigma^{j/2}\lb 0^{\ell-1}1^{\ell-1} \ldots (\kappa-1)^{\ell-1}\rb^\infty \ .
    \]
    This, however, implies that $x\neq \xi^{10\ldots 0}(y)$.  If $j$ is odd, then we obtain that
    \[
    \pi_\mathrm{odd}(x) \ = \ \sigma^{(j+1)/2}\lb \pi_\mathrm{even}\lb\xi^{0\ldots 0}(y)\rb\rb \ =
    \ \sigma^{(j+1)/2}\lb 0^{\ell-1}1^{\ell-1} \ldots (\kappa-1)^{\ell-1}\rb^\infty \ .
    \]
    If we had $x=\xi^{10\ldots 0}(y)$, then combined with (\ref{e.pd_lift_even_postions_2}) this would
    imply that the sequence $\xi^{10\ldots 0}(y)_{[0,+\infty)}$ is periodic, contradicting the
    aperiodicity of $y_{[0,+\infty)}$.  Altogether, this shows that $\xi^{10\ldots 0}(y)$ is not
    contained in the forward orbit of $\xi^{0\ldots 0}(y)$, so that $X$ is non-minimal.
\end{proof}

Finally, we treat the case $\ell=3$ and $\kappa\geq 2$ even. 
\begin{prop}
  If $\kappa\geq 2$ is even, then $X=\eta^{-1}_{\kappa,3}\lb X_\mathrm{pd}\rb$ is non-minimal.
\end{prop}
\begin{proof}
  As in the previous proof, we choose $y$ to be the period doubling sequence and consider the points
  $\xi^{00}(y)$ and $\xi^{10}(y)$. Similar to above, we have
  \begin{eqnarray}
    \label{e.pd_finalproof_1} \pi_\mathrm{even}\lb\xi^{00}(y)\rb & = & (01\ldots(\kappa-1))^\infty \ , \\
        \label{e.pd_finalproof_2} \pi_\mathrm{even}\lb\xi^{10}(y)\rb & = & (1\ldots(\kappa-1)0)^\infty \ . 
  \end{eqnarray}
  Since $y_k=0$ holds for all $k\in 4\N_0+1$, $\xi^{00}(y)$ is of the form
  \[
  \xi^{00}(y) \ = \ 0a_11a_12a_23a_2 \ldots a_\kappa(\kappa-1)a_\kappa0a_{\kappa+1}1a_{\kappa+1}\ldots \ ,
  \]
  with $a_i\in\cA_\kappa$ (where we assume $\kappa\geq 6$ for better illustration). At the same time, 
   \[
  \xi^{00}(y) \ = \ 1a_12a_13a_24a_2 \ldots a_{\kappa-1}(\kappa-1)a_\kappa0a_{\kappa}1a_{\kappa+1}\ldots \ . 
  \]
  Hence, if an odd number appears at an even position in $\xi^{00}(y)$, then its left and right
  neighbours coincide, and the same holds for even numbers at even positions in
  $\xi^{10}(y)$. Together with (\ref{e.pd_finalproof_1}) and (\ref{e.pd_finalproof_2}), this implies
  that $\xi^{10}(y)$ cannot belong to the forward orbit closure of $\xi^{00}(y)$. (A more formal
  argument can be given along the lines of the proof of
  Proposition~\ref{p.period_doubling_non-minimality}; we leave the details to the reader.)  This
  yields that $X$ is not minimal, as claimed.
\end{proof}

  \subsection{Decomposition of non-minimal lifts into minimal components}

  When a permutative lift of a minimal subshift is not minimal, we have the following.

  \begin{lem}
    Suppose that $h\in\cF(\kappa,\ell)$ is permutative, $Y\ssq\cA^\Z_\kappa$ is a minimal subshift and
    $X=\eta^{-1}_h(Y)$ is not minimal. Then for each minimal set $M\ssq X$ there exists an integer
    $r(M)\in\{1\ld \kappa^{\ell-1}-1\}$ such that, for all $y\in Y$, we have 
    \[
      \# M \cap \eta^{-1}_h(y) \ = \ r(M)\ . 
    \]
  \end{lem}
  We will refer to $r(M)$ as the {\em rank of $M$ over $Y$}.

  \proof Let $y,\hat y\in Y$ and assume that $\#\left(\eta^{-1}_h(y)\cap M\right)=m$ and
  $\eta^{-1}_h(y)\cap M=\left\{x^{(1)}\ld x^{(m)}\right\}$. By minimality of $Y$, compactness of $M$
  and continuity of $\eta_h$, we can choose $n_k\nearrow\infty$ such that
  $\kLim \sigma^{n_k}(y)=\hat y$ and
  $\kLim \sigma^{n_k}(x^{(i)})=\hat x^{(i)} \in\eta^{-1}_h(\hat y)\cap M$. Since the $x^{(i)}$ are
  pairwise $(\ell-1)$-separated by Proposition~\ref{p.intro_inverse_branches}, so are the limits
  points $\hat x^{(i)}$. This shows
  $\#\lb\eta^{-1}_h(\hat y)\cap M\rb \geq \#\lb\eta^{-1}_h(y)\cap M\rb$.  As $y,\hat y$ were
  arbitrary, we obtain that $\#\lb\eta^{-1}_h(y)\cap M\rb$ is independent of $y\in Y$.\qed\medskip

  However, it is important to note that the ranks of two distinct minimal subsets of the lift may well
  be different. An easy example is given by the lift of the fixed point $0^\Z$ under
  $\eta^{-1}_{2,3}$, which is given by the two fixed points $0^\Z$ and $1^\Z$ and the two-periodic
  orbit $\left\{(01)^\Z,(10)^\Z\right\}$. It turns out that all a priori possible combinations of
  ranks may occur. In order to see this, the following basic observation is needed.

\begin{lem}\label{l.permutative_block_code_via_a_return_permutation}
  Let $\kappa\in \N$, and $\pi$ be a permutation of $\cA_{\kappa}$. Then there exists a permutative
  block code $h\in \cF(\kappa,2)$ such that $\Phi^{(h)}_0 = \pi$.
\end{lem}
\begin{proof} It suffices to choose $h(ab)=b-\pi(a)\bmod \kappa$.
\end{proof}

Given a permutation $\pi$ of $\cA_\kappa$, we denote by $\cP(\pi)$ the partition of $\cA_\kappa$ into
$\pi$-orbits. Note that each element $a\in\cA_\kappa$ is periodic under $\pi$ with some period
$p(a)\leq\kappa$ and its orbit is given by $\cO_\pi(a)=\{\pi^i(a)\mid i=0\ld p(a)-1\}$. Now, choose
integers $1\leq r_1\ld r_n \leq \kappa$ such $\sum_{i=1}^n r_i=\kappa$. Let $\pi$ be a permutation of
$\cA_\kappa$ such that $\cP(\pi)$ can be written as $\cP(\pi)=\{P_1\ld P_n\}$ with $\#P_i=r_i$ for all
$i=1\ld \kappa$. Note that thus $p(a)=r_i$ for all $a\in P_i$.

Now, choose $h\in \cF(\kappa,2)$ such that $\Phi_0^{(h)}=\pi$. Then for any $a\in\cA_\kappa$, we
have that
\[
  \xi^a\lb 0^\Z \rb \ = \ \left( a\pi(a)\ldots \pi^{p(a)-1}(a)\right)^\infty
\]
which is a periodic point of period $p(a)$. As a consequence, $\eta^{-1}_h\lb 0^\Z\rb$ partitions into
$n$ different minimal sets (periodic orbits) $M_1\ld M_n$ , each of cardinality $\#
M_i=r(M_i)=r_i$. This provides an easy demonstration for the fact that in general there are no
combinatorial restrictions on how a permutative lift of a minimal subshift is partitioned into minimal
sets of different rank. More sophisticated examples based on Toeplitz flows will be given in
\cite{Kang2027Thesis}. These will be based on the following generalisation of
Theorem~\ref{t:intro_minimality}.

\begin{thm}
  Suppose that $h\in\cF(\kappa,\ell)$ is permutative, $Y\ssq\cA^\Z_\kappa$ is a minimal subshift and
  $X=\eta^{-1}_h(Y)$. Further, let $\cQ$ be a partition of $\cA^{\ell-1}_\kappa$ and assume that there
  exists a sequence of words \nfolge{w^{(n)}} such that the following hold (for all $n\in\N$ where
  applicable):
  \begin{itemize}
  \item[(i)] $|\tau^{(n)}| \nearrow +\infty$ as $n\to\infty$;
  \item[(ii)] $\tau^{(n+1)}$ starts with $\tau^{(n)}$;
  \item[(iii)] $\cR\lb\tau^{(n)}\rb$ leaves each $Q\in\cQ$ invariant, that is, $\pi(Q)=Q$ for all
    $\pi\in\cR\lb\tau^{(n)}\rb$;
  \item[(iv)] $\cR\lb\tau^{(n)}\rb$ acts transitively on each $Q\in\cQ$.
  \end{itemize}
  Then $X=\biguplus_{Q\in\cQ} M_Q$, where each $M_Q$ is a minimal set given by
  \begin{equation}
    M_Q \ = \ \left\{ x\in X\mid v\triangleleft x \follows \exists n\in\N,\ u\in \cQ: v\triangleleft \xi^u\lb
      w^{(n)}\rb\right\} \ . 
\end{equation}
Moreover, we have $r(M_Q)=\#Q$ for all $Q\in\cQ$.
\end{thm}
\begin{proof}
  Since $w^{(n+1)}$ starts with $w^{(n)}$, there exists a sequence $y\in Y$ such that
  $y_{[0,|w^{(n)}|-1]}=w^{(n)}$ for all $n\in\N$. Fix $\tilde u\in\cA^{\ell-1}_\kappa$ and let
  $M\ssq X$ be the minimal component that contains $\xi^{\tilde u}(y)$, that is,
  $M=\overline{\cO_\sigma(\xi^{\tilde u}(y))}$. Choose $n\in\N$ such that
  \begin{equation}\label{e.decomposition_minimal_sets}
    \cW_{|\omega^{(n)}|+\ell-1}(M)\cap \cW_{|\omega^{(n)}|+\ell-1}(M') \ = \ \emptyset \ ,
  \end{equation}
  where $M'=X\smin M$. Now, suppose that $u,\tilde u$ belong the same partition element $Q\in\cQ$ and
  there is $v\in\cA^*_\kappa$ such that $w^{(n)}vw^{(n)}\in\cW(Y)$ and $\Phi_{w^{(n)}v}(\tilde
  u)=u$. Then
  \[
    \xi^{\tilde u}\lb w^{(n)}vw^{(n)}\rb_{[|w^{(n)}v|,|w^{(n)}vw^{(n)}|-1]} \ = \ \xi^{u}\lb w^{(n)}\rb \ .
  \]
  Due to (\ref{e.decomposition_minimal_sets}), this entails that $\xi^{u}(y)\in M$. Similar as in the
  proof of Theorem~\ref{t:intro_minimality} in Section~\ref{Minimality_via_return_permutations}, the
  fact that $\cR\lb w^{(n)}\rb$ acts transitively on $Q$ then yields that $\xi^{u}(y)\in M$ for all
  $u\in Q$. Using minimality, we can conclude that $M=M_Q$.

  It remains to prove that the sets $M_Q$, $Q\in\cQ$, are pairwise disjoint. To that end, it suffices
  to show that $\xi^{u'}(y)\notin M_Q$ whenever $u'\notin Q$. Assume for a contradiction that
  $\xi^{u'}(y)\in M_Q$, so that in particular $\xi^{u'}\lb w^{(n)}\rb \in \cW\lb M_Q\rb$. By
  minimality, this yields $\xi^{u'}\lb w^{(n)}\rb \triangleleft \xi^{u}(y)$ for $u\in Q$. However,
  this is only possible if there exists some $v\in\cA^*_\kappa$ such that $y$ starts with
  $w^{(n)}vw^{(n)}$ and $\Phi_{w^{(n)}v}(u)=u'$. As $\Phi_{w^{(n)}v}\in\cR\lb w^{(n)}\rb$ in this
  case, this contradicts (iii).
\end{proof}

%%%%%%%%%%%%%%%%%%%%%%%%%%%%%%%%%%%%%%%%%%%%%%%%%%%%%%%%%%%%%%%%%%%%%%%%%%%%%%%%%%%%%%%%
%%%%%%%%%%%%%%%%%%%%%%%%%%%%%%%%%%%%%%%%%%%%%%%%%%%%%%%%%%%%%%%%%%%%%%%%%%%%%%%%%%%%%%%%
%%%%%%%%%%%%%%%%%%%%%%%%%%%%%%%%%%%%%%%%%%%%%%%%%%%%%%%%%%%%%%%%%%%%%%%%%%%%%%%%%%%%%%%%

\section{Unique ergodicity} \label{UniqueErgodicity}

%%%%%%%%%%%%%%%%%%%%%%%%%%%%%%%%%%%%%%%%%%%%%%%%%%%%%%%%%%%%%%%%%%%%%%%%%%%%%%%%%%%%%%%%

\subsection{Unique ergodicity of substitution subshifts} \label{UniqueErgodicityForSubstitutions}

Before we turn to discuss the unique ergodicity for lifts of arbitrary strictly ergodic
subshifts, we first want to point out in this section that for all the substitution examples discussed
in Section~\ref{SubstitutionApplications}, unique ergodicity is a more or less direct consequence of
minimality. The reason is the fact that the linear recurrence of primitive substitution subshifts is
preserved by lifts under permutative sliding block codes, in the sense that every minimal component of
the lift is linearly recurrent (and therefore uniquely ergodic).  This is true in general for all
primitive substitutions and will be discussed in
\cite{GohlkeJaegerKang2026PermutativeSubstitutionLifts}. However, in order to demonstrate the
underlying principle, we want to provide a simplified argument for the case of the noble means
substitutions. Recall that all primitive substitution subshifts are linearly recurrent
\cite{DamanikLenz2006SubstitutionSystemsAndLinearRepetitivity} and all linearly recurrent subshifts
are uniquely ergodic (see, for instance, \cite{Durand2000LinearRecurrence,Bruin2022SymbolicDynamics}).
\smallskip

Now, for given parameters $p\in\N$ and $\kappa,\ell\geq 2$, let
$X=\eta^{-1}_{\kappa,\ell}\lb X_{\mathrm{nm},p}\rb$. Due to the minimality of $X$ by
Theorem~\ref{t.noble_means_minimality}, there exists some $k\in\N$ such that for all
$u,u'\in\cA_{\kappa}^{\ell-1}$, the lift $\xi^u\lb \tau^{(k)}\rb$ of $\tau^{(k)}$ contains
$\xi^{u'}(0)$. Moreover, as shown in Section~\ref{NobleMeans}, the sequence of prefix permutations
$\nfolge{\Phi_{\tau^{(n)}}}$ is periodic with some period $q\in\N$. In particular, we have
$\Phi_{\tau^{(n-1)}} = \Phi_1$ and $\Phi_{\tau^{(n)}} = \Phi_0$ for all $n\in q\Z$. This immediately
implies that
\begin{equation}
  \label{e.ue_fibonacci_return_permutation_invariance}
  \Phi_{\Gamma_{\mathrm{nm},p}^n(w)} \ = \ \Phi_w
\end{equation}
holds for all $w\in\cA_\kappa^*$. Further, we therefore obtain that
\begin{equation}
  \label{e.ue_fibonacci_word_inclusion}
  \xi^{u'}\lb\tau^{(n)}\rb \triangleleft \xi^u\lb\tau^{(n+k)}\rb \textit{ holds for all }
  u,u'\in\cA^{\ell-1}_\kappa \textit{ and } n\in q\Z \ .
\end{equation}

Note also that for any $n\in\N$, any $y\in X_{\mathrm{nm},p}$ can be written as an infinite
concatenation of substitution words $\tau^{(n)}$ and $\tau^{(n-1)}$
\cite{BaakeGrimm2013AperiodicOrder}. This implies that any $w\in \cW\lb X_{\mathrm{nm},p}\rb$ with
$|w|\leq |\tau^{(n-2)}|$ is contained in $\tau^{(n)}$ as a subword.  \smallskip

In order to see that $X$ is linearly recurrent, let $\hat w=\xi^{\hat u}(w)\in \cW(X)$, with
$\hat u\in\cA^{\ell-1}_\kappa$ and $w\in \cW\lb X_{\mathrm{nm},p}\rb$.  Let $n\in q\Z$ be minimal with
the property that $|\tau^{(n-2)}| \geq |w|$, so that $w\triangleleft \tau^{(n)}$. Since we choose $n$
minimal, we have
\begin{equation}
  |\tau^{(n)}| \ \leq \ 2^{q+2}\cdot |w| \ . 
\end{equation}
Now, suppose that $v\in \cW\lb X_{\mathrm{nm},p}\rb$ satisfies $|v|\geq 2^{q+k+4}\cdot |w|$. Then
$|v|\geq 4|\tau^{(n+k)}|$, which implies $\tau^{(n+k)}\triangleleft v$. By
(\ref{e.ue_fibonacci_word_inclusion}), we obtain that
$ \xi^{u'}\lb \tau^{(n)}\rb \triangleleft \xi^u(v)$ holds for all
$u,u'\in\cA^{\ell-1}_\kappa$. Consequently, we also have $\xi^{u'}(w) \triangleleft \xi^u(v)$ for all
$u,u'\in\cA^{\ell-1}$, and in particular $\xi^{\hat u}(w) \triangleleft \xi^u(v)$. Since
$u\in\cA^{\ell-1}_\kappa$ and $v\in\cW\lb X_{\mathrm{nm},p}\rb$ with $|v|\geq 2^{q+k+4}$ were
arbitrary, we obtain that $\xi^{\hat u}(w)$ is contained in any word of length
$\geq 2^{q+k+4}\cdot |w|+\ell-1$ in $\cW(X)$. As $|\xi^{\hat u}(w)|=|w|+\ell-1$, the same holds in
particular for any word of length $\geq (\ell-1)\cdot 2^{q+k+4}\cdot |\xi^{\hat u}(w)|$. Therefore $X$
is linearly recurrent with constant $L=(\ell-1)\cdot 2^{q+k+4}$, and hence uniquely ergodic.\medskip

Together with Theorem~\ref{t.noble_means_minimality}, this proves Theorem~\ref{t:intro_fibonacci}. For
the Tribonacci and period doubling substitutions, we refer to the mentioned result in
\cite{GohlkeJaegerKang2026PermutativeSubstitutionLifts}.

%%%%%%%%%%%%%%%%%%%%%%%%%%%%%%%%%%%%%%%%%%%%%%%%%%%%%%%%%%%%%%%%%%%%%%%%%%%%%%%%%%%%%%%%

\subsection{Unique ergodicity via uniform frequencies of return permutations} \label{UniqueErgodicityViaReturnPermutations}

We now turn to the characterisation of unique ergodicity for general permutative lifts, without
relying on linear recurrence or any substitution structure. Thereby, the skew product representation
provided by Theorem~\ref{t:skew_product_representation} allows to see that there always exists a
canonical invariant Borel measure for the lift of a uniquely ergodic subshift by a permutative sliding
block code. The main task will then be to show the uniqueness of this particular invariant measure.

\begin{prop}\label{p.invariant_product_measure}
  Suppose that $Y\ssq\cA_\kappa^\Z$ is a subshift, $\mu_Y$ is a shift-invariant Borel measure on $Y$,
  $h\in\cF(\kappa,\ell)$ is a permutative block code and $X=\eta_h^{-1}(Y)$. Then there exists a
  shift-invariant Borel measure $\mu_X$ on $X$ which satisfies
  \begin{equation}
    \label{e.product_measure_1}
    \mu_X\left(\left[\xi^u(w)\right]^+_{|w|+\ell-1}\right) \ = \ \frac{1}{\kappa^{\ell-1}} \cdot
    \mu_Y\lb [w]^+_{|w|}\rb
  \end{equation}
  for all $w\in\cW(Y)$ and $u\in\cA^{\ell-1}_\kappa$ and is uniquely determined by these equations.
\end{prop}
\proof\quad First of all, we note that
\[
\cW(X) \ = \ \left\{\xi^u(w) \mid w\in \cW(Y),\ u\in\cA_\kappa^{\ell-1}\right\} \cup
\bigcup_{n=1}^{\ell-1} \cA^n_\kappa \ .
\]
Therefore, any Borel measure on $X$ is determined by its values on the cylinder sets of the form
$\left[\xi^u(w)\right]^+_{|w|+\ell-1}$. In particular, $\mu_X$ is uniquely determined by
(\ref{e.product_measure_1}).

In order to see that $\mu_X$ is indeed a shift-invariant Borel measure, note that the skew product
transformation $T:Y\times\cA^{\ell-1}_\kappa \to Y\times\cA^{\ell-1}_\kappa$ defined in
(\ref{e.skew_product_system}) has the invariant measure $\tilde\mu_Y=\mu_Y\times \theta$, where
$\theta$ denotes the equidistribution on $\cA^{\ell-1}_\kappa$. By definition, we have that
\begin{equation}
  \label{e.product_measure_2}
  \tilde\mu_Y\left([w]^+_{|w|}\times\{u\}\right) \ = \ \frac{1}{\kappa^{\ell-1}}\cdot
  \mu_Y\left([w]^+_{|w|}\right) \
\end{equation}
for any $w\in\cW(Y)$ and $u\in\cA^{\ell-1}_\kappa$.  Moreover, if $\gamma$ is the conjugacy between
$(X,\sigma)$ and $(Y\times\cA_\kappa^{\ell-1},T)$ given by (\ref{e.conjugacy_to_skew_product}), then
\begin{equation}
  \label{e.product_measure_3}
  \gamma^{-1}\left([w]^+_{|w|}\times\{u\} \right) \ = \ \left[\xi^u(w)\right]^+_{|w|+\ell-1} \ .
\end{equation}
Therefore, $\mu_X$ equals the measure $(\gamma^{-1})_*\tilde\mu_Y = \tilde\mu_Y\circ \gamma$ on
$X$. As $\gamma$ is a conjugacy and $\tilde\mu_Y$ is $T$-invariant, $\mu_X$ is a shift-invariant Borel
measure as claimed.  \qed\medskip

We say $x\in\cA_\kappa^\Z$ has {\em uniform word frequencies}\foot{Oxtoby \cite{Oxtoby1952ErgodicSets}
  calls this property {\em strict transitivity}, whereas in the context of aperiodic order and the
  study of Delone sets it is usually referred to as {\em uniform patch frequencies}
  \cite{Schlottmann1999GeneralizedModelSets}.}  if, for all $w\in\cA^*_\kappa$, there exist numbers
$\bar\nu(x,w)$ such that
\begin{equation}\label{e.uniform_word_frequencies}
  \lim_{n\to\infty} \sup_{k\in\Z} \left|\ntel \jnergsum \ind_{[w]^+_{|w|}}\circ \sigma^{k+j}(x)
  -\bar\nu(x,w)\right|\ = \ 0 \ .
\end{equation}
The following statement is well-known (see, for instance,
\cite{Oxtoby1952ErgodicSets,Schlottmann1999GeneralizedModelSets}).

\begin{prop}
  \label{p.unique_ergodicity_via_word_frequecies} If $Y\ssq \cA_\kappa^\Z$ is a minimal subshift, then
  the following are equivalent.  \romanlist
\item $(Y,\sigma)$ is uniquely ergodic;
   \item all $y\in Y$ have uniform word frequencies;
\item there exists some $y\in Y$ with uniform word frequencies.
 \listend
\end{prop}

\begin{remark}\label{r.unique_ergodicity_via_word_frequencies}
  If the supremum in (\ref{e.uniform_word_frequencies}) is only taken over $k\in\N$ (instead of
  $k\in\Z$), then this still yields unique ergodicity on the $\omega$-limit set of $y$. However, if
  $y$ is almost periodic, then the $\omega$-limit set equals the orbit closure. Hence, the statement
  of Proposition~\ref{p.unique_ergodicity_via_word_frequecies} remains valid if the supremum in
  (\ref{e.uniform_word_frequencies}) is only taken over $k\in\N$.
\end{remark}

In order to adapt this statement to our situation, we need some more notation. Given a permutative
block code $h\in\cF(\kappa,\ell)$, a strictly ergodic subshift $Y\ssq \cA_\kappa^\Z$, $w,v\in\cW(Y)$
with $|w|\leq |v|$ and $u,u'\in\cA_\kappa^{\ell-1}$, we define
\begin{eqnarray}\label{e.subword_appearences}
 \Ap(w,v) & = & \left\{ 0\leq j \leq |v|-|w|\mid v_{[j,j+|w|-1]}=w\right\} \eqand\\
  \label{e.subword_permutation_appearences}
   \Ap_{u\mapsto u'}(w,v) & = & \left\{ j\in \Ap(w,v)\mid \Phi_{v_{[0,j-1]}}(u)=u'\right\} \ .
\end{eqnarray}
Note that we keep the dependence of these quantities on $h$ implicit. Let further
\begin{eqnarray}
  \nu^\ap(w,v) & = & \frac{\#\Ap(w,v)}{|v|-|w|+1} \ , \\
  \nu^\ap_{u\to u'}(w,v) & = &  \frac{\#\Ap_{u\mapsto u'}(w,v)}{|v|-|w|+1} \ , 
\end{eqnarray}
as well as
\begin{eqnarray}
\label{e.subword_frequency_balance_from_prefix}
  D^\ap_u(w,v)  & = &  \max_{u',u''\in\cA^{\ell-1}} \left| \nu^\ap_{u\mapsto u'}(w,v)-\nu^\ap_{u\mapsto u''}(w,v)\right| \ ,  \\
  \label{e.subword_frequency_balance}
  D^\ap(w,v) & = & \max_{u\in\cA^{\ell-1}} D^\ap_u(w,v) \ . 
\end{eqnarray}
Then we obtain the following.
\begin{thm}
  \label{t:unique_ergodicity} Suppose that $Y\ssq\cA_\kappa^{\Z}$ is a strictly ergodic subshift with
  unique invariant probability measure $\mu_Y$ and $h\in\cF(\kappa,\ell)$ is a permutative block code
  of length~$\ell$.  Then $X=\eta_h^{-1}(Y)$ is strictly ergodic if and only if there exists $y\in Y$
  such that
  \begin{equation}\label{e.unique_ergodicity_criterion_1}
    \nLim \sup_{k\in\N} D^\ap\lb w,y_{[k,k+n]}\rb \ = \ 0
  \end{equation}
  holds for all $w\in\cW(Y)$. Moreover, (\ref{e.unique_ergodicity_criterion_1}) is equivalent to
   \begin{equation}\label{e.unique_ergodicity_criterion_2}
    \nLim \sup_{k\in\N}\max_{u,u'\in\cA_\kappa^{\ell-1}} \left(\nu^\ap_{u\mapsto u'}\lb w,y_{[k,k+n]}\rb
    - \frac{\mu_Y\lb [w]^+_{|w|}\rb}{\kappa^{\ell-1}}\right) \ = \ 0 \ . 
   \end{equation}
\end{thm}
\proof \quad We first show that unique ergodicity is equivalent to
(\ref{e.unique_ergodicity_criterion_2}). To that end, fix $\hat u\in\cA^{\ell-1}$ and $k\in\N$ and let
$u=\Phi_{y_{[0,k-1]}}(\hat u)$. Then, for any $u'\in\cA^{\ell-1}$, we have
\begin{eqnarray*}
  \lefteqn{ \ntel \jnergsum \ind_{[\xi^{u'}(w)]^+_{|w|+\ell-1}}\circ\sigma^{k+j}\lb \xi^{\hat u}(y)\rb \ =
  \ \ntel\jnergsum \ind_{[\xi^{u'}(w)]^+_{|w|+\ell-1}} \circ \sigma^j\lb \xi^u\lb\sigma^k(y)\rb \rb} 
    \\ & = & \nu^\ap_{u\mapsto u'}\lb w,y_{[k,k+n+|w|-2]}\rb \ \nKonv
           \ \frac{\mu_Y\lb[w]^+_{|w|}\rb}{\kappa^{\ell-1}} \ , \hspace{11eM}
\end{eqnarray*}
where (\ref{e.unique_ergodicity_criterion_2}) implies that the convergence is uniform in $k\in\N$.
Hence, the sequence $\xi^u(y)$ has uniform word frequencies and
Proposition~\ref{p.unique_ergodicity_via_word_frequecies} yields the unique ergodicity of
$(X,\sigma)$. Note that (\ref{e.unique_ergodicity_criterion_2}) in combination with the uniform
ergodic theorem implies that unique invariant measure must be the measure $\mu_X$ defined by
(\ref{e.product_measure_1}).\smallskip

It remains to show the equivalence between (\ref{e.unique_ergodicity_criterion_1}) and
(\ref{e.unique_ergodicity_criterion_2}). The fact that (\ref{e.unique_ergodicity_criterion_2}) implies
(\ref{e.unique_ergodicity_criterion_1}) is obvious. Conversely, in order to see that
(\ref{e.unique_ergodicity_criterion_1}) implies (\ref{e.unique_ergodicity_criterion_2}), note that the
unique ergodicity of $(Y,\sigma)$ yields
\[
\nu^\ap\lb w,y_{[k,k+n]}\rb \ \nKonv \ \mu_Y\lb [w]^+\rb \ ,
\]
and the convergence is uniform in $k\in\N$. However, we also have that
\[
\nu^\ap\lb w,y_{[k,k+n]}\rb \ = \ \sum_{u'\in\cA^{\ell-1}} \nu^\ap_{u\mapsto u'}\lb w,y_{[k,k+n]}\rb \ , 
\]
where $u\in\cA^{\ell-1}$ is arbitrary. In combination with (\ref{e.unique_ergodicity_criterion_1}),
this implies (\ref{e.unique_ergodicity_criterion_2}).\qed\medskip

In order to estimate the quantities $D^\ap(w,v)$, the following statement can be useful.

\begin{lem}\label{l.word_balance_concatenation}
  In the situation of Theorem~\ref{t:unique_ergodicity}, suppose that $v,w\in\cW(Y)$ and $v$ is a
  concatenation of words $v=v^{(0)}\ldots v^{(M)}$ and $|w|\leq |v^{(j)}|$ for all $j=0\ld M$. Then
  \begin{equation} \label{e.word_balance_concatenation}
    D^\ap(w,v) \ \leq \ \sum_{j=0}^M \frac{|v^{(j)}|-|w|+1}{|v|-|w|+1} D^\ap\lb w,v^{(j)}\rb +
    \frac{M\cdot(|w|-1)}{|v|-|w|+1} \ . 
  \end{equation}
\end{lem}
\proof\quad Let $k_i=\sum_{j=0}^{i-1}|v^{(j)}|$, so that $v_{[k_i,k_{i+t}-1]}=v^{(i)}$. Then, since 
\[
\Phi_{v_{[0,k_i+k]}}\ = \ \Phi_{v^{(i)}_{[0,k]}}\circ \Phi_{v_{[0,k_i-1]}}
    \]
    for all $k=0\ld |v^{(i)}|-1$, we have
    \[
    \Ap_{u\mapsto u'}(w,v) \cap \left[k_i,k_{i+1}-|w|\right] \ = \ \Ap_{\Phi_{v_{[0,k_i-1]}}(u)\mapsto
      u'}\lb w,v^{(i)}\rb - k_i \ . 
    \]
    Combined with the definition of $D^\ap(v,w)$, this yields
    (\ref{e.word_balance_concatenation}). \qed\medskip

% \begin{remark}
% %  \label{r.ue_criterion_length_2}
%   If we restrict to block codes $h_{\kappa,2}$ in our example family, then as discussed in
%   Section~\ref{ExampleFamily}, the return pertubation for any word $w\in\cW(Y)$ is given by a rotation 
%   \[
%   \Phi_w \ = \ R_{t(w)} : a  \mapsto a+t(w)\bmod \kappa
%   \]
%   on the cyclic group $\cA_\kappa\simeq\Z_\kappa$, where $t(w)=\sum_{j=0}^{|w|-1} w_j$. In particular,
%   all the return perturbations commute with any rotation $R_b$ on $\cA_\kappa$. This immediately
%   implies that $\nu^\ap_{a\to a'}(w,v)=\nu^\ap_{a+b\to a'+b}$ holds for all
%   $a,a',b\in\cA_\kappa$. Therefore $D^\ap_a(w,v)$ is independent of $a$ and we have
%   $D^\ap_a(w,v)=D^\ap(w,v)$ for all $a\in\cA_\kappa$.
  %This greatly simplifies the proof of unique
  %ergodicity and the application of Theorem~\ref{t:unique_ergodicity} in this situation, and we will
  %make use of this fact in the next section.
%\end{remark}

\section{Permutative sliding block codes and Toeplitz flows}\label{ToeplitzFlows}

\subsection{Basic notions for Toeplitz flows}

A point $\xi \in \cA^{\Z}$ is called a {\em Toeplitz sequence} if it is not periodic and has the
property that
\begin{equation}\label{e.toeplitz_condition_two-sided}
	\textit{for all $n\in \Z$ there exists $p\in \N$ such that ${\xi}_{n} = {\xi}_{n+pk}$ holds
	for any $k\in \Z$.}
\end{equation}
By $p(\xi,n)$ we denote the least integer $p$ that satisfies (\ref{e.toeplitz_condition_two-sided}).
Given a point $\xi \in \cA_\kappa^{\Z}$, $a\in\cA_\kappa$ and any $p\in \N$, we let
$\Per(\xi,p,a) = \{n\in \Z\ |\ {\xi}_{n+kp} = a \ \textrm{for all}\ k\in \Z \}$, and define the set of
$p$-periodic positions of $\xi$ as
\[
  \Per(\xi,p) \ = \ \bigcup_{a\in\cA_\kappa} \Per(\xi,p,a) \ .
\]
The set of all periodic positions of $\xi$ is then given as $\Per(\xi) = \bigcup_{p\in \N}\Per(\xi, p)$.
We call $p\in \N$ an {\em essential period} of $\xi$ if $\Per(\xi,p)\neq \varnothing$ and there exists
no $q<p$ such that $\Per(\xi,p,a)\subseteq \Per(\xi,q,a)$ holds for all $a\in\cA_\kappa$.

\begin{remark}
  Note that if $\Per(\xi,p,a)\subseteq\Per(\xi,q,a)$, then we actually have
  $\Per(\xi,p,a)=\Per(\xi,\gcd(p,q),a)$, where $\gcd$ denotes the greatest common
  divisor. Hence, in order to see that $p$ is essential, it suffices that
  $\Per(\xi,p,a)\neq\Per(\xi,q,a)$ for all $q<p$ with $q|p$.
\end{remark}

For any
$\xi\in \cA_{\kappa}^{\Z}$ and $p\in \N$, we further define the $p$-skeleton $\hat{\xi}^p \in
\cA_{\kappa+1}^{\Z}$ of $\xi$ by
\begin{equation}\label{e.p_skeletons}
	\hat{\xi}^{p}_n \ = \ \begin{cases} \xi_n\ \textrm{if}\ n\in
        \Per(\xi,p)\\ \kappa\ \textrm{if}\ n\notin \Per(\xi,p)\end{cases} \ , \quad n\in\Z \ . 
\end{equation}
Then $\hat{\xi}^p$ is $p$-periodic and $p$ is an essential period of $\xi$ if it is the least
period of $\hat{\xi}^{p}$.  A strictly increasing sequence of positive integers $P = (p_k)_{k\in \N}$
is called a {\em scale} if $p_k \mid p_{k+1}$ for all $k\in \N$.  It is called a {\em period
  structure} of a Toeplitz sequence $\xi \in \cA_\kappa^{\Z}$ if
\begin{itemize}
	\item for all $k\in \N$, $p_k$ is an essential period of $\xi$;
	\item $\bigcup_{k\in \N} \Per(\xi, p_k) = \Z$.
\end{itemize}

Given a Toeplitz sequence $\xi \in \cA_\kappa^{\Z}$, we call $X_{\xi}$ a {\em Toeplitz subshift}.  We
have
\begin{thm}[\cite{JacobsKeane1969ToeplitzSequences}]
	Every Toeplitz subshift is minimal.
\end{thm}

For a Toeplitz sequence $\xi\in \cA_\kappa^{\Z}$ and a period structure $P = (p_k)_{k\in \N}$ of $\xi$, we
define the density of periodic positions of $\xi$ with respect to $P$ as
\[
	\overline D(\xi, P) \ = \  \lim_{k\rightarrow \infty} \frac{\# \Per(\xi, p_k)\cap [0,p_k-1]}{p_k}.
\]
Note that $\frac{\# \Per(\xi, p_k)\cap [0,p_k-1]}{p_k}$ is non-decreasing in $k$ and bounded from
above by~$1$, so that the limit always exists (and can be replaced by the supremum).  Suppose
$Q=(q_k)_{k\in \N}$ is another period structure of $\xi$. Then for any $k\in \N$, there exists
$N\in \N$ such that $p_k\mid q_m$ and $q_k\mid p_m$ for $m\geq N$. Hence
$\overline D(\xi,P)=\overline D(\xi,Q)$ and we can define the density of $\xi$ as
$\overline D(\xi) = \overline D(\xi,P)$, where $P$ is any period structure of~$\xi$.  A Toeplitz
sequence $\xi\in \cA_\kappa^{\Z}$ is called {\em regular} if $\overline D(\xi) = 1$ and {\em
  irregular} if $\overline D(\xi)<1$.  A Toeplitz subshift $X\subseteq \cA_\kappa^{\Z}$ is then called
regular if it contains a regular Toeplitz sequence and irregular otherwise.

\begin{thm}[\cite{JacobsKeane1969ToeplitzSequences}]
	A regular Toeplitz subshift is uniquely ergodic.
\end{thm}

%%%%%%%%%%%%%%%%%%%%%%%%%%%%%%%%%%%%%%%%%%%%%%%%%%%%%%%%%%%%%%%%%%%%%%%%%%%%%%%%%%%%%%%%%%%%%%%%%%%%%

\subsection{Toeplitz subshifts from one-sided Toeplitz sequences}

Examples of Toeplitz subshifts are often constructed via one-sided Toeplitz sequences. Since this is
usually done only implicitly in the literature, we provide some details for the convenience of the
reader. In analogy to (\ref{e.toeplitz_condition_two-sided}), a point $\xi\in \cA_\kappa^{\N_0}$ is
called a {\em one-sided Toeplitz sequence} if it is not periodic and has the property that
\begin{equation} \label{e.toeplitz_property_one-sided}
\begin{split}	& \textit{for all $n\in \N_0$, there exists $p\in \N$ such that ${\xi}_{n} = {\xi}_{n+pk}$}\\ & \qquad \textit{holds
    for any $n\in \Z$ with $n+pk\geq 0$.}
\end{split}
\end{equation}
By $p(\xi,n)$, we denote the least integer $p$ that satisfies (\ref{e.toeplitz_property_one-sided}) for a given $n\in \N_0$ .
Given $\xi\in \cA_\kappa^{\N_0}$, we again define its associated subshift by $X_{\xi}$.  When $\xi$ is
almost periodic, we alternatively can choose any $z\in \cA_\kappa^{\Z}$ with $z_{[0,+\infty]} = \xi$
and define $X_\xi$ as the $\omega$-limit of $z$.

Given a one-sided sequence $\xi\in \cA_\kappa^{\N_0}$ and any $p\in \N$, we define the $p$-periodic
positions of $\xi$ as ${\Per}^{+}(\xi, p) = \left\{n\in \N_0\ | \ p(\xi, n)|p \right\}$.  Again, $p$
is an essential period of $\xi$ if $\Per^+(\xi, p)\neq \varnothing$ and there exists no $q<p$ such
that $\Per^+(\xi,p,a)\subseteq \Per^+(\xi, q,a)$ for any $a\in\cA_\kappa$. If we define the
$p$-skeleton $\hat\xi^p$ of a one-sided Toeplitz sequence as in (\ref{e.p_skeletons}), with
$\Per(\xi,p)$ replaced by $\Per^+(\xi,p)$, then this means that $p$ is essential if and only if it is
the least period of $\hat\xi^p$.  Similar to before, a scale $P = (p_k)_{k\in \N}$ is a period
structure of a Toeplitz sequence $\xi \in \cA_\kappa^{\N_0}$ if
\begin{itemize}
	\item for all $k\in \N$, $p_k$ is an essential period of $\xi$;
	\item $\bigcup_{k\in \N} {\Per}^+(\xi, p_k) = \N_0$.
\end{itemize}
The density of the periodic positions of $\xi$ can then be defined as above, with $\Per(\xi,p_k)$
replaced by $\Per^+(\xi,p_k)$, and will be denoted by $\overline D^+(\xi,P)$ and $\overline D^+(\xi)$,
respectively. Then $\xi$ is called regular if $\overline D^+(\xi) = 1$ and irregular
otherwise.\smallskip

We omit the proof of the following elementary statement.

\begin{lem}\label{l.toeplitz_essential_period}
  If $X\subseteq \cA_{\kappa}^{\Z}$ is a minimal subshift and $\xi,\zeta\in X$.  Then, for any
  $p\in \N$, there exists $q\in \{0\ld p-1\}$ such that $\Per(\zeta, p,a) = \Per(\xi, p,a) + q$ holds
  for all $a\in\cA_\kappa$. In particular, we have $\Per(\zeta,p)=\Per(\xi,p)+q$.  Moreover, if
  $\xi, \zeta$ are Toeplitz sequences, then any period structure of $\xi$ is also a period structure
  of~$\zeta$ and vice versa.
\end{lem}

As a consequence, we can speak of a period structure of a Toeplitz subshift.

\begin{prop}\label{p.period_structure_of_toeplitz_subshift}
  If $\xi\in \cA_\kappa^{\N}$ a one-sided Toeplitz sequence, then
	\begin{itemize}
	\item $X_{\xi}$ is a Toeplitz subshift;
	\item if $P$ is a period structure of $\xi$, then $P$ is a period structure of $X_{\xi}$;
	\item $X_{\xi}$ is regular if and only if $\xi$ is regular.
	\end{itemize}
\end{prop}

\begin{proof}
	Let $\xi \in \cA_{\kappa}^{\N_0}$ be a one-sided Toeplitz sequence with period structure $P =
        (p_k)_{k\in \N}$.  We first note that since $\xi$ is almost periodic, the subshift $X_\xi$ is
        minimal.

        In order to see hat $P$ is a period structure for $X_\xi$, we have to show that it is a
        period structure for some Toeplitz sequence $x\in X_\xi$. To that end, let \jfolge{n_j} and
        \jfolge{r_j} be defined by recursion on $j$ as follows. For $j=0$ we let $n_0=0$ and
        $r_0 = p(\xi, 0)$. If $n_j, r_j$ are given for some $j\geq 0$, we let $n_{j+1} = n_j + r_j$
        and
        \[
        r_{j+1} \ = \ 
        \operatorname{lcm}\left\{r_j, p(\xi, n_{j+1} -j-1), p(\xi, n_{j+1}+j+1)\right\} \ ,
        \]
        where $\operatorname{lcm}$ denotes the least common multiple. By induction on $j$, one sees
        that all positions in $[n_j-j,n_j+j]$ are $r_j$-periodic within $\xi$. Hence, we have that
        \[
        \xi_{[n_{j+1} -j, n_{j+1}+j]} \ = \ \xi_{[n_j+r_j-j,n_j+r_j+j]} \ = \  \xi_{[n_j-j, n_j+j]} \ .
        \]
        As a consequence, we obtain a well-defined sequence $x\in X_{\xi}$ by setting
        \begin{equation}\label{e.toeplitz_construction_from_one-sided}
        x_{[-j,j]} \ =  \ \xi_{[n_j-j, n_j+j]}
        \end{equation}
        for all $j\in\N$.  If we choose $z\in \cA_{\kappa}^{\Z}$ with $z_{[0,\infty)} = \xi$, then it
        follows by construction that $\lim_{j\to \infty}\sigma^{n_j}(z) = x \in X_\xi$. We claim that
        $x$ is Toeplitz.

        By going over to a subsequence if necessary, there exist $a_k\in \{0\ld p_k-1\}$, $k\in\N$,
        such that $n_j \bmod p_k \equiv a_k$ for all but finitely many $j$. Consequently,
          \[
            {\Per}^+(\xi, p_k) -a_k + p_k \Z \ \subseteq \ \Per(x, p_k) \ .
            \]
            Due to (\ref{e.toeplitz_construction_from_one-sided}), the backwards iterates
            $\sigma^{-n_j}(x)$ converge to a point that coincides with $\xi$ on $\N_{0}$ as
            $j\to \infty$. Hence, we have
            \[
            \left(\Per(x, p_k)+a_k\right)\cap \N_0 \  \subseteq \ {\Per}^+(\xi, p_k) \ .
            \]
            As $p_k$ is an essential period of $\xi$, this implies that it is also an essential period
            of $x$. Since $P$ is a period structure of $\xi$, we have that for any $j\in \N$ there
            exists $k\in \N$ such that $r_j \mid p_k$. Hence
            \[
            \Z \ = \ \bigcup_{j\in \N}\Per(x,
            r_j) \ \subseteq \ \bigcup_{k\in p_k}\Per(x, p_k) \ ,
            \]
            so that $x$ is indeed a two-sided Toeplitz sequence, with period structure $P$. As
            $X_{\xi}$ is minimal, we have $X_{\xi} = \overline{\{\sigma^n(x)\ |\ n\in \Z\}}$. Thus, we
            obtain that $X_{\xi}$ is a Toeplitz subshift. Moreover, it follows from the above
            construction that
            \[
              \#{\Per}(x, p_k)\cap [0,p_k-1] \ = \ \#{\Per}^+(\xi, p_k)\cap [0, p_k-1]
            \]
            for all $k\in\N$. Therefore $x$ is regular if and only if $\xi$ is a regular.
\end{proof}

\begin{remark}
  We note that while in the situation of Proposition~\ref{p.period_structure_of_toeplitz_subshift},
  $X_{\xi}$ is a Toeplitz flow, the sequence $\xi$ may not necessarily extend to a two-sided Toeplitz
  sequence. This is the case, for instance, for the period doubling sequence.
\end{remark}

%%%%%%%%%%%%%%%%%%%%%%%%%%%%%%%%%%%%%%%%%%%%%%%%%%%%%%%%%%%%%%%%%%%%%%%%%%%%%%%%%%%%%%%%%%%%%%%%%%%%%

\subsection{Construction of one-sided Toeplitz sequences} \label{ToeplitzConstruction}

A standard method to construct a one-sided Toeplitz sequence in $\cA^{\N_0}_\kappa$ is to enlarge the
alphabet $\cA_\kappa$ by an additional letter (which in our case will be $\kappa$) and to approximate
$\xi$ with $p_k$-periodic sequences $\xi^{(k)}\in\cA^{\N_0}_{\kappa+1}$ in such a way that $\xi^{(k)}$
and $\xi^{(k+1)}$ only differ at positions $n\in\N_0$ with $\xi^{(k)}_n=\kappa$. Thus, the additional
letter $\kappa$ serves as a placeholder in the construction, which is successively replaced by letters
from the original alphabet. We refer to \cite{Williams1984ToeplitzFlows} or
\cite{Downarowicz2005ToeplitzFlows} and references within for more background and the construction of
a wide array of specific examples.\smallskip

In our context, we will implement this idea in the following way. Given a scale $P = (p_k)_{k\in \N}$
with $p_1>1$, we choose a sequence of words $v^{(k)}\in \cA_{\kappa+1}^{p_k}$ of length $p_k$ and
consider the corresponding $p_k$-periodic sequences $\xi^{(k)}\in \cA_{\kappa+1}^{\N_0}$ defined by
$\xi^{(k)}_{[0, p_k-1]} = v^{(k)}$. Thereby, we require that the following conditions hold:
\begin{itemize}
\item[(T1)] $\xi^{(k+1)}_{j} = \xi^{(k)}_{j}$ whenever $\xi^{(k)}_j \neq \kappa$;
\item[(T2)] for any $k\in \N$ there exists $k'>k$ such that $\xi^{(k')}_{[0, p_k)} \in \cA_{\kappa}^{p_k}$;
\item[(T3)] if $\xi^{(k)}_j = \kappa$ for some $k\in \N$ and $j\in\N_0$, then there exists $k'>k$ and
  $m\in \N$ such that $\xi^{(k')}_{j}, \xi^{(k')}_{j + mp_k} \in \cA_{\kappa}$ and $\xi^{(k')}_j\neq
  \xi^{(k')}_{j+m p_k}$;
\item[(T4)] for any $k\in\N$, $p_k$ is the least period of $\xi^{(k)}$.
\end{itemize}

The precise role of these conditions will become clear in the proof of
\begin{prop}\label{p.toeplitz_construction}
  Suppose that (T1)--(T4) hold. If we let
  \[
  k(n) \ = \ \min\left\{k\in\N\mid \xi^{(k)}_n\neq \kappa\right\}
  \]
  and define $\xi \in\cA_\kappa^\N$ by $\xi_n \ = \ \xi^{(k(n))}_{n}$, then $\xi$ is a one-sided
  Toeplitz sequence.  Further, the scale $P$ used in the construction is a period structure of $\xi$,
  the $p_k$-skeletons $\hat\xi^{p_k}$ coincide with the sequences $\xi^{(k)}$ for all $k\in\N$ and we
  have
     \[
     \overline D^+(\xi) \ = \ \lim_{k\to\infty} \frac{\# \{0\leq j < p_k-1\mid v^{(k)}_j \neq
       \kappa\}}{p_k} \ .
     \]
\end{prop}
\begin{proof}
  Note that due to (T2), we have that $k(n)$ is finite for all $n\in\N_0$, and by $(T1)$ we obtain
  that $\xi^{(k)}_n=\xi^{(k(n))}_n$ for all $k\geq k(n)$. Further, we have $n\in\Per^+(\xi,p_{k(n)})$
  by construction, while (T3) yields $n\notin \Per^+(\xi,p_{k'})$ for any $k'<k(n)$. Consequently
  \[ {\Per}^{+}(\xi,p_k) \ = \ \left\{ n\in\N_0\mid \xi^{(k)}_n\neq \kappa\right\}
      \]
      and therefore $\hat\xi^{p_k} = \xi^{(k)}$.  Since $p_k$ is the least period of $\xi^{(k)}$ by
      (T4), this implies that $P$ is a period structure of $\xi$. In particular, we obtain that $\xi$
      is Toeplitz, since every symbol is repeated with some period $p_k$ and $\xi$ is aperiodic since
      the $p_k$ are not uniformly bounded.
\end{proof}
\begin{remark}
  We note that (T4) can be ensured by requiring the stronger condition
\begin{itemize}
\item[(T4')] for every $k\in\N$ and $q<p_k$ with $q|p_k$, there exist $n,n'\in\N$ such that
  $q|(n-n')$, $\xi^{(k)}_n,\xi^{(k)}_{n'}\neq \kappa$ and $\xi^{(k)}_n\neq \xi^{(k)}_{n'}$.
\end{itemize}
In this case, the fact that $p_k$ is the least period of $\xi^{(k)}$ is already ensured by those
positions with $\xi^{(k)}_n\neq \kappa$. We will use this observation in some constructions below.

\end{remark}

%%%%%%%%%%%%%%%%%%%%%%%%%%%%%%%%%%%%%%%%%%%%%%%%%%%%%%%%%%%%%%%%%%%%%%%%%%%%%%%%%%%%%%%%%%%%%%%%%%%%%

\subsection{Minimality of permutative lifts of Toeplitz flows}

We now aim to explore the relations between Toeplitz constructions and the lifting procedure via
permutative block codes. Our first observation is the following.

\begin{thm}\label{t.toeplitz_with_minimal_lifts}
  For any $\ell,\kappa\in \N$ there exists a Toeplitz flow $Y\subseteq \cA_\kappa^{\Z}$ such that for
  any permutative block code $h\in \cF(\kappa, \ell)$ the lift $X=\eta_h^{-1}(Y)$ is minimal.
\end{thm}
\begin{proof}

  Fix any permutative block code $h\in\cF(\kappa,\ell)$. Consider a scale $P = (p_k)_{k\in \N_0}$ with
  $p_1 > \ell-1$, $q_k = \frac{p_{k+1}}{p_k}> \kappa^{\ell-1}$ for any $k\in \N_0$. By
  recursion on $k$, we construct two sequences of words $w^{(k)}$ and $v^{(k)}=w^{(k)}\kappa^{\ell-1}$
  of length $p_k-\ell+1$ and $p_k$, respectively, so that the $p_k$-periodic sequences $\xi^{(k)}$
  with $\xi^{(k)}_{[0,p_k-1]}=v^{(k)}$, $k\in\N$, satisfy properties (T1)--(T4) above.

We first choose a word $w^{(1)}\in \cA_{\kappa}^{p_0 - \ell +1}$ such that (T4) is satisfied for
$v^{(1)} = w^{(1)}\kappa^{\ell-1}$. Then, if $v^{(k)}=w^{(k)}\kappa^{\ell-1}$ is already given for
some $k\geq 1$, we choose $u^{(k,0)}\ld u^{(k,q_k-2)}\in \cA_\kappa^{\ell-1}$ such that
\begin{equation}\label{e.toeplitz_minimal_lift_construction}
  \left\{u^{(k,i)} \mid i=0\ld \kappa^{\ell-1}-1\right\}\ = \ \cA_{\kappa}^{\ell-1} \ 
\end{equation}
(and $u^{(k,\kappa^{\ell-1})}\ld u^{(k,q_k-2)}$ are arbitrary). Now, we define
\[
w^{(k+1)} \ = \ w^{(k)}u^{(k,0)}\ld w^{(k)}u^{(k,q_k-2)}w^{(k)} \eqand v^{(k+1)} = w^{(k+1)}\kappa^{\ell-1} \ . 
\]
Then (T1) and (T2) hold by construction. Further, $\xi^{(k+1)}$ differs from $\xi^{(k)}$ only on
intervals of the form $[mp_k-(\ell-1),mp_k-1]$ for $m\in\Z\smin q_k\Z$. Since we insert all the
different words in $\cA^{\ell-1}_\kappa$ into these intervals by
(\ref{e.toeplitz_minimal_lift_construction}), none of the newly filled positions can be
$p_k$-periodic. This yields (T3). Since $\kappa$ only appears in the last $\ell-1$ positions in
$v^{(k)}$, $p_k$ is the least period of $\xi^{(k)}$. This shows (T4).  As a consequence,
Proposition~\ref{p.toeplitz_construction} yields that the sequence $\xi$ resulting from this
construction is a one-sided Toeplitz sequence with period structure $P$. If we let $Y=X_\xi$, it
remains to see that the lift of $Y$ is minimal. However, it is a direct consequence of
(\ref{e.toeplitz_minimal_lift_construction}) and Lemma~\ref{l:independence_lemma} that $\cR(w^{(k)})$
acts transitively on $\cA_{\kappa}^{\ell-1}$ for each $k\in \N$. Hence, $\eta_h^{-1}(X_{\xi})$ is
minimal by Theorem~\ref{t:intro_minimality}.
\end{proof}

\begin{remark}
  The Toeplitz flows resulting from the above construction will be regular. However, with some
  straightforward modifications, it is also possible to obtain irregular examples.
\end{remark}

As the above result shows, a Toeplitz flow may admit a minimal extension when lifted by any
permutative sliding block code of a fixed length. However, the situation changes when the length of
the permutative block code is allowed to vary. The following result shows that no Toeplitz flow can
have a minimal lift under all permutative sliding block codes on the same alphabet.

\begin{thm} \label{t.Toeplitz_fixed_h_non-minimal}
	Consider a Toeplitz flow $Y\subseteq \cA_{\kappa}^{\Z}$.  Then there exist some $\ell\in\N$ and
        a permutative block code $h\in\cF(\kappa,\ell)$ such that $\eta_h^{-1}(Y)$ is not minimal.
\end{thm}
\begin{proof}
  We fix a Toeplitz sequence $y\in Y$, let $\ell=\min\{p(y,n)\mid n\in\Z \}+1$, choose
  $n\in\Per(y,\ell-1)$, let $c=y_n$ and consider the permutative block code $h\in\cF(\kappa,\ell)$
  given by
\[
h : \cA^{\ell}_\kappa\to\cA_\kappa \quad, \quad w_0\ldots w_{\ell-1}\mapsto w_{\ell-1}-w_0+c \bmod
\kappa \ .
\]
Note that given any lift $\xi^u(y)$, $u\in\cA^{\ell-1}_\kappa$, this implies
$\xi^u(y)_{k+\ell-1}=\xi^u(y)_k+y_k-c$ for any $k\in \Z$. Since $\ell-1$ is the least period that occurs for any symbol
in $y$, we obtain that
\[
\Per\lb \xi^u(y),\ell-1\rb \ = \ \Per(y,\ell-1,c) \ .
\]
Further, given $a\in\cA_\kappa$, we obtain
\[
  \Per\lb\xi^{a\ldots a}(y),\ell-1\rb \ = \ \Per\lb y,\ell-1,c\rb \ ,
\]
and due to Lemma~\ref{l.toeplitz_essential_period} we have
$\Per\lb x,\ell-1\rb = \Per\lb x,\ell-1,c\rb$ for any $x$ in the orbit closure of
$\xi^{a\ldots a}(y)$. However, this implies that the orbit closures of $\xi^{a\ldots a}(y)$ and
$\xi^{b\ldots b}(y)$ with $a\neq b$ are disjoint, so that $\eta_h^{-1}(Y)$ is not minimal.
\end{proof}

The following result shows that for any fixed permutative sliding block code, there exists a Toeplitz
flow with non-minimal lift.

\begin{prop} \label{p.h_fixed_non-minimal_Toeplitz}
  Let $h\in\cF(\kappa,\ell)$ be permutative. Then there exists a Toeplitz subshift $Y\subseteq \cA_{\kappa}^{\Z}$ such that
  $X=\eta^{-1}_h(Y)$ is non-minimal.
\end{prop}
\proof We again use the same construction scheme as before and choose a scale \nofolge{p_k} with
$p_1>\ell$ and let $q_k=\frac{p_{k+1}}{p_k}$. Further, we let
\[
  w^{(1)} \ = \ 1^{2\ell}0^{2\ell} \eqand v^{(0)}\ = \ w^{(1)}\kappa^\ell \ .
\]
Note that our construction scheme then implies that the word $w^{(1)}$ appears in the resulting
sequence $\xi$ (defined by $\xi_{[0,p_k-1]}=v^{(k)}$) exactly at positions in $p_1\N_0$.  Fix some
$u\in\cA^{\ell-1}_\kappa$. Then Lemma~\ref{l:independence_lemma} implies that there exist two prefixes
$u^{(0)},u^{(1)}\in\cA_\kappa^{\ell-1}$ such that
\[
  \Phi_{w^{(1)}0u^{(0)}}(u) \ = \ \Phi_{w^{(1)}1u^{(1)}}(u) \ = \ u \ .
\]
We now recursively define the words $w^{(k)}$ and $v^{(k)}$ of our construction scheme by
\[
  w^{(k+1)} \ = \
  \begin{cases}
    \ \lb w^{(k)}0u^{(0)}\rb^{q_k-1}w^{(k)}\kappa^\ell & \ \textrm{ if } k \textrm{ is odd}\\
        \ \lb w^{(k)}1u^{(1)}\rb^{q_k-1}w^{(k)}\kappa^\ell & \ \textrm{ if } k \textrm{ is even}\\
  \end{cases} \ . 
\]
Then the fact that conditions (T1)--(T4) hold is straightforward to check. Hence, the construction
yields a Toeplitz subshift $Y=X_\xi$. Moreover, we obtain that $\Phi_{w^{(1)}vw^{(1)}}(u)=u$ holds
whenever $w^{(1)}vw^{(1)}\in\cW(Y)$. This implies that $\cR\lb w^{(1)}\rb$ does not act transitively
on $\cA_\kappa^{\ell-1}$, so that $X$ is non-minimal by Theorem~\ref{t:intro_minimality}.
\qed\medskip

Finally, we show the existence of a Toeplitz flow with universally non-minimal lifts.

\begin{thm} \label{t.Toeplitz_with_universally_non-minmal_lifts} There exists a regular Toeplitz flow
  $Y\subseteq \cA_{2}^{\Z}$ such that for any $\kappa,\ell\geq 2$ and any permutative block code
  $h\in \cF(\kappa, \ell)$, the lift $X = \eta_h^{-1}(Y)$ is not minimal.
\end{thm}
\proof Suppose that $(h_n)_{n\in \N}$ is an enumeration of all permutative block codes in
$\bigcup_{\kappa,\ell\geq 2} \cF(\kappa,\ell)$. Since any permutation of a finite set is periodic, and
for any permutative block code $h\in \cF(\kappa, \ell)$ the set
$\left\{\Phi^{(h)}_w\ \middle|\ w\in \cA_{\kappa}^{*}\right\} \subseteq
\Per\lb\cA_{\kappa}^{\ell-1}\rb$ is finite, there exists $q=q(h)$ such that
\[
  \Phi^{(h)}_{w^q} \ = \ \lb \Phi^{(h)}_w\rb^q \ = \ \Id_{\cA_\kappa^{\ell-1}} \quad \textrm{ for all } 
    w\in\cA_\kappa^*.
\]
We call such a $q$ a period of $h$. For each $n\in \N$, we fix a period $q_n\geq 2$ of
$h_n$.\smallskip

We now use the same construction scheme as before in order to obtain the desired Toeplitz sequence
without minimal permutative lifts.  Fix any $q_0\geq 2$. Consider a scale $P = (p_n)_{n\in \N}$ with
$p_1 = 3q_0$ and $\frac{p_{n+1}}{p_n} = 3q_n$. For $i\in\{0,1\}$, define a block code
$g^{(i)}\in \cF(3,1)$ by
\[
  g^{(i)}(a)\ = \ \begin{cases} a, & \text{if } a\in \cA_2\\ i, & \text{if } a = 2\end{cases}.
\]
We denote again by $g^{(i)}$ its natural extension to $\cA_{3}^k$ for any $k\in \N$. Note that thus
applying the block code $g^{(i)}$ to a word in $\cA_{3}^*$ replaces every occurrence of the letter $2$
by $i$ and leaves the rest unchanged. We let
\[
  w^{(1)} \ = \ 01^q2^q0^{q-1}\ \eqand \ v^{(1,i)} \ = \ g^{(i)}(w^{(1)}),\ i\in \{0,1\}\ .
\]
Further, we recursively define
\[
w^{(n+1)} \ = \ v^{(n,0)}{v^{(n,1)}}^{q_n}{w^{(n)}}^{q_n}{v^{(n,0)}}^{q_n-1}
\]
and let $v^{(n+1, i)} = g^{(i)}(w^{(n+1)})$ for $i\in \{0,1\}$. It is straightforward to check that
the $p_n$-periodic sequences $\xi^{(n)}$ with $\xi^{(n)}_{[0, p_n-1]} = v^{(n)}$, $n\in \N$, satisfy
conditions (T1)-(T3) above. Suppose that $d\leq p_n$ with $d \mid p_n$ is another period of
$\xi^{(n)}$. Then $d\geq \frac{2p_n}{3}$, because $2$ only appears within a block of length
$\frac{p_n}{3}$ in $v^{(n)}$, and hence $d=p_n$. Thus (T4) holds. Moreover, by an easy induction, the
frequency of occurrences of $2$ in $\xi^{(n)}$ is $\frac{1}{3^n}$. Therefore
Propositions~\ref{p.toeplitz_construction} and \ref{p.period_structure_of_toeplitz_subshift} yield
that $Y=X_\xi$ is a regular Toeplitz subshift. By construction, we also obtain that for any $n\in\N$
the sequence $\xi$, and hence any sequence in $Y=X_\xi$, is a concatenation of words $v^{(n,0)}$ and
$v^{(n,1)}$.  \smallskip

For any $w\in \cW(Y)$, define the set of first return words to $w$ as 
\[
  \cW_\mathrm{ret}(w) = \left\{v\in \cW(Y)\ \left| \begin{array}{c} vw\in \cW(Y) \text{ starts with } 
        w \text{ and}\\\text{$w$ appears exactly twice in $vw$}
                                                   \end{array}
                                                 \right. \right\}\ .
\] We show by induction that
\begin{equation}\label{e.return_words_to_inflations_of_01}
  \cW_\mathrm{ret}\lb v^{(n,0)}v^{(n,1)}\rb  \ = \ \left\{v^{(n+1, 0)}, v^{(n+1 , 1)}\right\}
\end{equation}
For $n =1$, observe that $\left\{j = 0\ld 3q_0-2\ \middle|\ v^{(1,i)}_{[j,j+1]} = 01\right\} = \{0\}$
and $v^{(1,i)}_0 = 0$ for $i\in\{0,1\}$, hence
$\cW_\mathrm{ret}(01) = \left\{v^{(1,0)},\ v^{(1,1)}\right\}$. Since $01$ is a prefix of $v^{(1,0)}$,
any $v\in \cW_\mathrm{ret}(v^{(1,0)}v^{(1,1)})$ is tiled by $v^{(1,0)}$ and $v^{(1,1)}$, i.e.
$v = s^{(0)}\ldots s^{(k)}$ where $s^{(j)} \in \{v^{(1,0)}, v^{(1,1)}\}$ for $j=0\ld k$. By
construction, we have 
\[
  \left\{j = 0\ld 3q_1-2 \ \middle|\ v^{(2,i)}_{[jp_1,(j+2)p_1-1]} = v^{(1,0)}v^{(1,1)}\right\} =
  \{0\},\
\]
and $v^{(1,0)}v^{(1,1)}$ cannot in the exact middle of $v^{(2,i)}v^{(2,i')}$ for any
$i,i'\in \{0,1\}$, because $v^{(2,i)}$ starts with $v^{(1,0)}$ for every $i\in \{0,1\}$. Therefore, we
must have $\cW_\mathrm{ret}\lb v^{(1,0)}v^{(1,1)}\rb=\left\{v^{(2,0)},v^{(2,1)}\right\}$, and thus
(\ref{e.return_words_to_inflations_of_01}) holds for $n=1$.

Now assume $\cW_\mathrm{ret}(v^{(n,0)}v^{(n,1)}) \ = \ \left\{v^{(n+1, 0)},v^{(n+1 , 1)}\right\}$. By
definition, $v^{(n+1,0)}$ begins with $v^{(n,0)}v^{(n,1)}$. Further, as in the case $n=1$, the
construction implies that $v^{(n+1,0)}v^{(n+1,1)}$ can occur in $\cW(Y)$ only as a prefix of
$v^{(n+2,i)}$, $i\in\{0,1\}$. This shows \eqref{e.return_words_to_inflations_of_01}.\smallskip

In order to complete the proof, observe that for every $n\in \N$ we have
\begin{equation}\label{e.prefix_permutations_of_inflation_words}
  \begin{split}
    \Phi^{(h_n)}_{v^{(n+1,i)}} &  = \ \lb \Phi^{(h_n)}_{v^{(n,0)}}\rb^{q_n-1}\circ
                                 \underbrace{\lb \Phi^{(h_n)}_{v^{(n,i)}}\rb ^{q_n}}_{=\Id_{\cA_\kappa^\ell-1}} \circ
                                 \underbrace{\lb \Phi^{(h_n)}_{v^{(n,1)}}\rb^{q_n}}_{=\Id_{\cA_\kappa^\ell-1}}
                                 \circ \Phi^{(h_n)}_{v^{(n,0)}} \\
                               & = \ \lb \Phi^{(h_n)}_{v^{(n,0)}}\rb^{q_n} \ = \  \Id_{\cA_\kappa^\ell-1}\ . 
  \end{split}  
\end{equation}
For $w\in \cW(Y)$, recall that $\wh{\cR}_{h_n}^Y(w)=\left\{\Phi^{(h_n)}_v\ \middle|\ v\in \cW_\mathrm{ret}(w)\right\}$. Then, by \eqref{e.return_words_to_inflations_of_01} and \eqref{e.prefix_permutations_of_inflation_words},
\[
  \left\langle\wh{\cR}_{h_n}^Y(v^{(n,0)}v^{(n,1)})\right\rangle \ = \ {\Id_{\cA_\kappa^\ell-1}}
\]
holds for every $n\in\N$. Since $\kappa,\ell\geq 2$, $\wh{\cR}_{h_n}^Y(w^{(n,0)}w^{(n,1)})$ cannot act
transitively. Due to Theorem~\ref{t:intro_minimality}, with $\cR^Y_{h_n}(w)$ replaced by
$\wh\cR_{h_n}^Y$ according to Remark~\ref{r.alternative_return_permutation_sets}, we obtain that
$\eta_{h_n}^{-1}(Y)$ is not minimal.  \qed\medskip

%%%%%%%%%%%%%%%%%%%%%%%%%%%%%%%%%%%%%%%%%%%%%%%%%%%%%%%%%%%%%%%%%%%%%%%%%%%%%%%%%%%%%%%% 

\subsection{Unique ergodicity of permutative lifts of Toeplitz flows}

We first show that unique ergodicity is not a direct consequence of minimality in our setting. This
can already be demonstrated by using the simple block code $h_{2,2}$ from Section~\ref{ExampleFamily}.

\begin{thm} \label{t.Toeplitz_with_minimal_but_non-uniquely_ergodic_lift}
  There exists a strictly ergodic Toeplitz subshift $Y\ssq\cA^\Z_2$ such that $\eta_{2,2}^{-1}(Y)$ is
  minimal, but not uniquely ergodic. 
\end{thm}
\proof\quad We use the construction theme introduced in Section~\ref{ToeplitzConstruction} and first
choose
\[
w^{(1)} \ = \ 110^{p_1-3} \ \eqand \ v^{(1)} \ = \ w^{(1)}2 \ .
\]
with $p_1\geq 6$, so that $|w^{(1)}|=p_1-1$ and $|v^{(1)}|=p_1$.  Note that according procedure
described in Section~\ref{ToeplitzConstruction} and due to the specific choice of $v^{(1)}$, the word
$w^{(1)}$ occurs exactly at positions $n\in p_1\Z$ in the resulting sequence $\xi\in\cA^\Z_2$:
\begin{equation}
  \label{e.minimal_nonunique_construction_firstproperty}
  \xi_{[n,n+p_1-2]}\ = \ w^{(1)} \quad \equi \quad n\in p_1\Z \ . 
\end{equation}
Further, we fix a sequence of integers \kfolge{q_k} with $q_k\geq 3$ and recursively define
\[
w^{(k+1)} \ = \ \lb w^{(k)}0\rb^{q_k-3}w^{(k)}1w^{(k)}1w^{(k)} \eqand v^{(k+1)} \ = \ w^{(k+1)}2 \ .
\]
Then conditions (T1)--(T4) with scale \kfolge{p_k} given by $p_k=p_1\cdot\prod_{j=1}^{k-1} q_j$ are
straightforward to check, so that Propositions~\ref{p.toeplitz_construction} and
\ref{p.period_structure_of_toeplitz_subshift} yields that $Y=X_\xi$ is a regular Toeplitz subshift.
Moreover, the following statements hold:
\begin{itemize}
\item[(i)] the words $w^{(k)}$ contain an even number of $1$s (follows by induction on $k$);
\item[(ii)] consequently, we have
  \[
    \Phi_{\lb w^{(k)}0\rb^j} \ = \ \Id_{\cA_2} \ = \ \Phi_{\lb w^{(k)}0\rb^{q_k-3}w^{(k)}1w^{(k)}1}
    \ ;
    \]
    for all $k\in\N$ and $j=1\ld q_k-3$;
  \item[(iii)] at the same time
    \[
          \Phi_{\lb w^{(k)}0\rb^{q_k-3}w^{(k)}1} \ = \ \bar {\ \cdot \ } \ ,
   \]
   where $\bar{\ \cdot \, }:\cA_2\to\cA_2, a\mapsto 1-a$, and this permutation belongs to
   $\cR\lb w^{(k)}\rb$.
 \end{itemize}
 By construction, we have $\xi_{[0,|w^{(k)}|-1]} = w^{(k)}$.  In particular, any word $w\in\cW(Y)$
 is contained in $w^{(k)}$ for some $k\in\N$. Therefore (iii), in combination with
 Theorem~\ref{t:intro_minimality}(iii), yields the minimality of $X = \eta_{2,2}^{-1}(Y)$. \smallskip

In order to see that $X$ is not uniquely ergodic, note that due
(\ref{e.minimal_nonunique_construction_firstproperty}) and properties (ii) and (iii), we have
\begin{equation}
  \label{e.minimal_nonunique_starting_estimate}
  D^\ap\lb w^{(1)},w^{(2)}\rb \ = \ \frac{q_1-2}{p_1(q_1-1)+1} \ 
\end{equation}
as well as 
\begin{eqnarray*}
  \lefteqn{  \left(\#\Ap_{0\mapsto 0}\lb w^{(1)},w^{(k+1)}\rb - \#\Ap_{0\mapsto1}\lb w^{(1)},w^{(k+1)}\rb \right) }\\
  & = & (q_k-1) \cdot \left(\#\Ap_{0\mapsto 0}\lb w^{(1)},w^{(k)}\rb - \#\Ap_{0\mapsto1}\lb w^{(1)},w^{(k)}\rb \right)
  \\ & &   - \left(\#\Ap_{0\mapsto 0}\lb w^{(1)},w^{(k)}\rb - \#\Ap_{0\mapsto1}\lb w^{(1)},w^{(k)}\rb \right) \\
  & = &  (q_k-2) \cdot \left(\#\Ap_{0\mapsto 0}\lb w^{(1)},w^{(k)}\rb - \#\Ap_{0\mapsto1}\lb w^{(1)},w^{(k)}\rb \right) \ . 
\end{eqnarray*}
This entails that 
\begin{eqnarray*}
  D^\ap\lb w^{(1)},w^{(k+1)}\rb & = & (q_k-2) \cdot \frac{|w^{(k)}|-|w^{(1)}|+1}{w^{(k+1)}-w^{(1)}+1} \cdot D^\ap\lb w^{(1)},w^{(k)}\rb
  \\ &  \geq & \frac{(q_k-2)\cdot (q_{k-1}-1)}{q_k\cdot q_{k-1}} \cdot D^\ap\lb w^{(1)},w^{(k)}\rb \ . 
\end{eqnarray*}
holds for all $k\geq 2$. As a consequence, we obtain that for all $k\in\N$
\[
  D^\ap\lb w^{(1)},w^{(k+1)}\rb \ \geq \ \left( \prod_{j=2}^{k}
  \frac{(q_j-2)\cdot (q_{j-1}-1)}{q_j\cdot q_{j-1}}\right) \cdot \frac{q_1-2}{p_1(q_1-1)+1}  \ . 
\]
Clearly, we can now choose the parameters $p_1$ and $q_k$ with $k\in\N_0$ such that
\[
  \inf_{k\in\N} D^\ap\lb w^{(1)},w^{(k)}\rb \ > \ 0 \ . 
\]
Due to Theorem~\ref{t:unique_ergodicity}, this shows that $X=\eta_{2,2}^{-1}(Y)$ cannot be uniquely
ergodic.  \qed\medskip

Our final aim is to modify the construction of the Toeplitz sequences in the proof of
Theorem~\ref{t.toeplitz_with_minimal_lifts} in order to obtain not only minimal, but strictly ergodic
lifts.

\begin{thm}\label{t:unique_ergodic_lift}
  Let $h\in \cF(\kappa,\ell)$ be permutative. Then there exists a Toeplitz subshift
  $Y\subseteq \cA_{\kappa}^{\Z}$ such that $X= \eta_h^{-1}(Y)$ is strictly ergodic.
\end{thm}
To prove this statement, we first recall a classical result from finite graph theory. A finite
directed graph is a pair $D=(V,E)$, where $V$ is a finite set of vertices and $E\subseteq V\times V$
is a set of directed edges. It is called \emph{connected} if for every $v,w\in V$ there are vertices
$v=v_0,v_1,\ldots,v_n=w$ such that $(v_i,v_{i+1})\in E$ for $i\in \{0\ld n-1\}$. For $v\in V$, define
the {\em indegree} of $v$ as $\deg^-(v)=\#\{(u,v)\in E\}$, and the {\em outdegree} of $v$ as
$\deg^+(v)=\#\{(v,u)\in E\}$. The graph is \emph{balanced} if $\deg^-(v)=\deg^+(v)$ for every
$v\in V$.  An \emph{Euler circuit} in $D$ is a finite sequence of edges $(v_0,v_1)\ld(v_{n-1},v_n)$ in
which every edge of $D$ appears exactly once.  We will use the following standard criterion for the
existence of an Euler-circuit.
\begin{thm}\label{t:Euler-circuit}\cite[Theorem~1.7.2]{BangJensenGutin2009Digraphs}.
A finite directed graph admits an Euler circuit if it is connected and balanced.
\end{thm}
For our construction, it will be useful to consider colored edges. An {\em edge-colored directed
  graph} is a directed graph together with a finite set of colors $C$ and a coloring map
$\chi:E\to C$. An Euler circuit of an edge-colored case will be denoted as
\[
v_0,\chi(v_0,v_1),v_1,\chi(v_1,v_2),\ldots,v_{n-1},\chi(v_{n-1},v_n)  , 
\]
thus keeping track of the colors.  Further, we will consider the case where $G$ is a finite group and
$S\subseteq G$. The \emph{left Cayley graph} of $G$ induced by $S$ has vertex set $G$ and the set of
edges $\{(g, sg)\ |\ g\in G,\ s\in S\}$.  Every vertex has indegree and outdegree $|S|$, so this graph
is balanced. If $S$ generates $G$, then it is connected.  Given a finite set $A$, we denote by
$\Per(A)$ the group of permutations of $A$.

Now we are ready to prove Theorem~\ref{t:unique_ergodic_lift}. As mentioned, we will use a refined
version of the construction in Theorem~\ref{t.toeplitz_with_minimal_lifts} and mainly concentrate on
the required modifications.
\begin{proof}[\textit{\bfseries Proof of Theorem~\ref{t:unique_ergodic_lift}}]
  Fix a permutative block code $h\in\cF(\kappa,\ell)$.  In order to start the inductive construction,
  we first choose a word $w^{(1)}\in \cA_{\kappa}^*$ and let $v^{(1)} = w^{(1)}\kappa^{\ell-1}$ and
  $p_1 = |v^{(1)}|$. Now, suppose that the words $v^{(1)}\ld v^{(k)}$ with
  $v^{(j)}=w^{(j)}\kappa^{\ell-1}$ and $|v^{(j)}|=p_j$ have already been defined.  Due to
  Lemma~\ref{l:independence_lemma}, the mapping
  \[
    \pi_k: \cA_{\kappa}^{\ell-1}\to \Per\lb\cA_{\kappa}^{\ell-1}\rb\quad , \quad u\mapsto
    \Phi_{w^{(k)}u}
  \]
  is injective. Consider the finite set of permutations
\[
	S^{(k)} = \left\{ \Phi_{w^{(k)}u}\ \middle|\ u\in \cA_{\kappa}^{\ell-1}\right\}\ .
\]
Denote by $G_k$ the group generated by $S^{(k)}$.  Consider the left Cayley graph $(G_k,E_k)$
of $G_k$ induced by $S_k$ (so that $E_k=\{(g,sg)\mid g\in G_k, s\in S^{(k)}\}$), with the edge
$(g, \Phi_{w^{(k)}u}\circ g)$ colored by $u$. In other terms, we choose $\cA_{\kappa}^{\ell-1}$ to be
the set of colors and use the mapping $\chi(g,h) = \pi_k^{-1}(hg^{-1})$ as the coloring map. We set
\[
q_k = \sum_{g\in G_k} \deg^+(g) = \kappa^{\ell-1}\cdot \# G_k\ .
\]
The graph $(G_k,E_k)$ is connected and balanced, so Theorem~\ref{t:Euler-circuit} gives an Euler
circuit. We assume without loss of generality that it starts at the identity and write it as
\begin{equation}\label{e:Euler circuit}
	g^{(k, 0)}, u^{(k,0)}\ld g^{(k, q_k-1)}, u^{(k, q_k-1)}
\end{equation}
where $g^{(k,0)}=\Id_{\cA_\kappa^{\ell-1}}\in \Per\lb\cA_{\kappa}^{\ell-1}\rb$ and
$\chi\left( g^{(k,i)}, g^{(k,i+1)}\right) = u^{(k, i)}$ for $i \in \{0\ld q_k-2\}$.  Now choose
\[
	w^{(k+1)} \ = \ w^{(k)}u^{(k,0)}\ldots w^{(k)}u^{(k,q_k-2)}w^{(k)}\quad  , \quad v^{(k+1)} \ = \  w^{(k+1)}\kappa^{\ell-1}
\]
and $p_{k+1} = p_k\cdot q_k$. Note that here we did not use the last edge in the Euler circuit.

The verification of (T1)--(T4) is the same as in Theorem~\ref{t.toeplitz_with_minimal_lifts}. We
obtain a one-sided Toeplitz sequence $\xi$ with period structure $P=(p_k)_{k\in\N}$ and set
$Y=X_\xi$. Since the density of holes in the $p_k$-skeleton is $\frac{\ell-1}{p_k}\to0$, the Toeplitz
flow $Y$ is regular and hence strictly ergodic. Similar to the situation in
Theorem~\ref{t.toeplitz_with_minimal_lifts}, $S^{(k)}$ acts transitively on $\cA_{\kappa}^{\ell-1}$
and $S^{(k)}\subseteq \cR(w^{(k)})$. Since $|w^{(k)}|\to \infty$,
Theorem~\ref{t:intro_minimality}(iii) shows that $X = \eta_h^{-1}(Y)$ is minimal.

Now we show that $X$ is uniquely ergodic. It follows by construction that
\[
	\Phi_{w^{(k)}u^{(k, 0)}\ldots w^{(k)}u^{(k, i)}} = g^{(k, i+1)}
\]
for $i \in \{0 \ld q_k-2\}$. Transitivity of $S^{(k)}$ acting on $\cA_\kappa^{\ell-1}$ in combination with the
orbit-stabilizer theorem give that, for any $v, v' \in \cA_{\kappa}^{\ell-1}$,
\[
	\# \left\{g\in G_k\ \middle|\ g(v) = v'\right\} = \frac{\#G_k}{\kappa^{\ell-1}}\ .
\]
Meanwhile, since every $g\in G_k$ occurs $\kappa^{\ell-1}$ times among the source vertices of the
Euler circuit \eqref{e:Euler circuit}, it follows that
\begin{equation}\label{e.vertex_balance}
\#\left\{i\in \{0\ld q_k-1\}\ \middle|\ g^{(k,i)}(v) = v'\right\} = \#G_k
\end{equation}
for any $v, v' \in \cA_{\kappa}^{\ell-1}$. Now, in order to apply Theorem~\ref{t:unique_ergodicity},
choose $y\in Y$ with $y_{[0,\infty)}=\xi$, and fix $w\in\cW(Y)$. Take $k$ such that
$|w|\leq|w^{(k)}|$. Consider $n, m\in \N$. We aim to estimate the quantity $D^\ap(w, y_{[n, n+m-1]})$
as defined in \eqref{e.subword_frequency_balance}.  First, we consider the aligned $p_{k+1}$-blocks.
Let $n_1, n_2$ be the smallest and largest integers in $[n, n+m-1]$, respectively, that are a multiple
of $p_{k+1}$.  Then by construction
\[
  y_{[n_1, n_2-1]} = w^{(k+1)}v_1w^{(k+1)}v_2\ldots w^{(k+1)}v_{t} \ ,
\]
where $v_1\ld v_t\in \cA_{\kappa}^{\ell-1}$ and $t\leq \floor{\frac{m}{p_{k+1}}}$. Suppose that
$w\triangleleft w^{(k)}$, that is, $w^{(k)}_{[j, j+|w|-1]} = w$ for some $j\leq
|w^{(k)}|-|w|$. Consider any prefix $u\in \cA_{\kappa}^{\ell-1}$. Then the prefixes
$\xi^{u}(y)_{[k,k+\ell-2]}$ that appear in $\xi^{u}(y)$ at positions $k=j,j+p_k\ld j+(q_k-1)p_{k}$
(over the corresponding $p_k$-separated occurrences of $w$ in $w^{(k+1)}$) are
\[
\Phi_{w^{(k)}_{[0,j-1]}}\left(g^{(k,i)}(u)\right),\qquad i= 0\ld q_k-2,
\]
where $\Phi_{w^{(k)}_{[0,j-1]}}=\Id_{\cA^{\ell-1}_\kappa}$ for $j=0$. Due to \eqref{e.vertex_balance},
these prefixes are uniformly distributed over $\cA_{\kappa}^{\ell-1}$, independent of the prefix
$u$. Thus, all occurrences of $w$ contained entirely in one copy of $w^{(k+1)}$ do not contribute to
$D^\ap\lb w,y_{[n,n+m-1]}\rb$. While occurrences of $w$ at other positions may contribute, the
interval $[n_1,n_2-1]$ contains only $tq_k\cdot(|w|+\ell-2)-|w|$ such positions. If we add the at most
$2p_{k+1}$ remaining positions in $[n,n+m-1]\smin[n_1,n_2-1]$, this yields
\begin{eqnarray*}
  D^{\ap}(w, y_{[n, n+m-1]})  &  \leq & \frac{tq_k\cdot(|w|+\ell-2)+2p_{k+1}-|w|}{m-|w|} \\ & \leq &
 \frac{tq_k\cdot (|w|+\ell)+2p_{k+1}}{tp_{k+1}} \ = \ \frac{|w|+\ell}{p_k} + \frac{2}{t} \ \kKonv \ 0 \ . 
\end{eqnarray*}
If follows from Theorem~\ref{t:unique_ergodicity} that $X$ is uniquely ergodic. 
\end{proof}
\begin{remark}
  The Toeplitz flow constructed above is actually linearly recurrent, but this is not
  essential. Indeed, the numbers $q_k=\kappa^{\ell-1}\#G_k$ are uniformly bounded because $G_k$ is a
  subgroup of the finite group $\Per\lb\cA_{\kappa}^{\ell-1}\rb$. At the $k$-th stage, however, we may
  traverse the chosen Euler circuit $c_k$ times before constructing $w^{(k+1)}$. This would result in
  multipliers $q_k=c_k\kappa^{\ell-1}\#G_k$ in the Toeplitz construction, without affecting the
  balance argument.  By choosing $c_k$ to grow sufficiently fast, one obtains similar examples that
  are not linearly recurrent.
\end{remark}

{\bfseries AI Declaration.}\quad AI tools were used for literature research and proof-reading/minor
editorial corrections. \medskip

%\section{Dynamical spectra of lifted subshifts}

%\subsection{Toeplitz flows}\ \bigskip

%{\bf Notation:} Always use $\alpha,\beta,\gamma\in\cA$ for single symbols, $a,b,c\in\cA_\kappa^{n-1}$ for
%pre- and suffixes and $u,v,w$ for longer words. Adapt notation above and below.
%\section{Symmetry of lifts}
%	\input{symmetry2}
%\section{Minimality of lifts}
%  \input{minimality2}
  
%\section{Unique ergodicity of lifts}
%\input{uniqueergodicity}

%\bibliographystyle{alpha} \bibliography{tj-references}

\newcommand{\etalchar}[1]{$^{#1}$}

\end{document}